\documentclass[11pt,reqno,a4paper]{amsart}
\usepackage{amsmath,amsfonts,amssymb,amsthm,comment}
\usepackage[abbrev,lite,nobysame]{amsrefs}
\usepackage{mathrsfs}
\usepackage{graphics,graphicx}
\usepackage[usenames,dvipsnames]{color}
\usepackage{times}
\usepackage{bbm}
\usepackage{enumerate,enumitem}
\usepackage{mathtools}
\usepackage[margin=1in]{geometry}
\usepackage[colorlinks=true, pdfstartview=FitV, linkcolor=blue, citecolor=blue, urlcolor=blue]{hyperref}

\makeatletter

\renewcommand\subsection{\@startsection{subsection}{2}%
\normalparindent{.5\linespacing\@plus.7\linespacing}{-.5em}
{\normalfont\bfseries}}
\renewcommand\subsubsection{\@startsection{subsubsection}{3}%
\normalparindent{.5\linespacing\@plus.7\linespacing}{-.5em}
{\normalfont\bfseries}}

\newcommand{\diampar}[1]{\vspace{.5em}\noindent $\diamond$ \normalfont {\itshape #1.}}

\def\@tocline#1#2#3#4#5#6#7{\relax
  \ifnum #1>\c@tocdepth 
  \else
    \par \addpenalty\@secpenalty\addvspace{#2}%
    \begingroup \hyphenpenalty\@M
    \@ifempty{#4}{%
      \@tempdima\csname r@tocindent\number#1\endcsname\relax
    }{%
      \@tempdima#4\relax
    }%
    \parindent\z@ \leftskip#3\relax \advance\leftskip\@tempdima\relax
    \rightskip\@pnumwidth plus4em \parfillskip-\@pnumwidth
    #5\leavevmode\hskip-\@tempdima
      \ifcase #1
       \or\or \hskip 1em \or \hskip 2em \else \hskip 3em \fi%
      #6\nobreak\relax
    \dotfill\hbox to\@pnumwidth{\@tocpagenum{#7}}\par
    \nobreak
    \endgroup
  \fi}
\makeatother

\newtheorem{proposition}{Proposition}[section]
\newtheorem{lemma}[proposition]{Lemma}
\newtheorem{theorem}[proposition]{Theorem}
\newtheorem{corollary}[proposition]{Corollary}

\theoremstyle{definition}

\newtheorem{remark}[proposition]{Remark}

\newcommand\eps{\varepsilon}
\newcommand\e{{\rm e}}
\renewcommand{\d}{\,\mathrm{d}}
\newcommand{\p}{\rho}
\renewcommand{\th}{d}
\newcommand{\pv}{\p_\eps}
\newcommand{\uv}{u_\eps}

\newcommand{\bpv}{\bar{\p}_{\eps}}
\newcommand{\buv}{\bar{u}_{\eps}}

\newcommand{\huv}{\hat{u}_\eps}
\newcommand{\intx}{\int_0^1}
\newcommand{\intt}{\int_0^t}
\newcommand{\intT}{\int_0^T}
\newcommand{\ddt}{\frac{\d}{\d t}}

\newcommand{\dis}{\displaystyle}
\newcommand{\ex}{\mathbb{E}}
\newcommand{\bu}{\bar{u}}

\newcommand{\bp}{\bar{\p}}

\newcommand{\cE}{\mathcal{E}}

\newcommand{\supt}{\sup_{0 \leq t \leq T}}

\newcommand{\tp}{\rho}
\newcommand{\tu}{u}
\newcommand{\rv}{r_\eps}
\newcommand{\vv}{v_\eps}

\newcommand{\bXv}{\bar{\mathcal{X}}_\eps}
\newcommand{\X}{\mathcal{X}}
\newcommand{\bmv}{\bar{\mu}_\eps}
\newcommand{\bX}{\bar{\mathcal{X}}}
\numberwithin{equation}{section}

\begin{document}
\title[Vanishing Mach Number Limit of Stochastic Compressible Flows]{Vanishing Mach Number Limit of \\ Stochastic
Compressible Flows}

\author[G.-Q. Chen]{Gui-Qiang G. Chen }
\address{Mathematical Institute, University of Oxford, Oxford OX2 6GG, UK}
\email{gui-qiang.chen@maths.ox.ac.uk}

\author[M. Coti Zelati]{Michele Coti Zelati}
\address{Department of Mathematics, Imperial College London, London, SW7 2AZ, UK}
\email{m.coti-zelati@imperial.ac.uk}

\author[C.~C. Yeung]{Chin Ching Yeung}
\address{Mathematical Institute, University of Oxford, Oxford OX2 6GG, UK}
\email{yeung.chinching@maths.ox.ac.uk}

\subjclass[2020]{35R60, 76N06, 37L40, 	35B25}

\keywords{Stochastic compressible Navier-Stokes equations, zero-Mach limit, acoustic systems, invariant measures}

\begin{abstract}
We study the vanishing Mach number limit for the stochastic Navier-Stokes equations
with $\gamma$-type pressure laws,
with focus on the one-dimensional case.
We prove that, if the stochastic term vanishes with respect to the Mach number sufficiently fast, the deviation
from the incompressible state of the solutions (for $\gamma \geq 1$) and the invariant measures (for $\gamma = 1$)
is governed by a linear stochastic acoustic system in the limit.
In particular, the critically sufficient decay rate for the stochastic term is slower than the corresponding
results with deterministic external forcing due to the martingale structure of the noise term,
and the blow-up of the noise term for the fluctuation system can be allowed.
\end{abstract}
\maketitle

\setcounter{tocdepth}{2}
\tableofcontents

\section{Introduction}

We are concerned with the vanishing Mach number limit of the stochastic Navier-Stokes
equations with $\gamma$-type pressure laws:
\begin{equation}\tag{NS}\label{eq:NS}
\begin{cases}
\dis\d\p+(\p u)_x \d t = 0,\\[1mm]
\dis\d(\p u) + (\p u^2 + \frac{1}{\eps^2}\p^\gamma)_x \d t = u_{xx} \d t +\p\sigma_\eps \d W_t,
\end{cases}
\end{equation}
where $\p=\p(t, x)$ and $u=u(t,x)$ represent the density and velocity of the fluid
in the position $x\in(0,1)$ at time $t \geq 0$, respectively.
The system is complemented with homogeneous Dirichlet boundary conditions:
\begin{align}\label{BC-0}
u(t, 0) = u(t, 1) = 0  \qquad\mbox{ for all $t >0$},
\end{align}
and initial data:
\begin{align}\label{ID-0}
\p(0, x) = \p_0(x), \quad u(0, x) = u_0(x)  \qquad\text{ for all $x\in (0,1)$}.
\end{align}
The conservation of mass guaranteed by the continuity equation in system \eqref{eq:NS} allows to fix,
without loss of generality, the initial mass to $1$, implying that
\begin{equation}
\intx \p(x, t) \d x = 1  \qquad\text{ for all $t >0$}.
\end{equation}
The constant $\eps>0$ represents the Mach number of the flow, which is (proportional to) the ratio of the typical flow
speed and the speed of sound propagating in the flow.
Without forcing, in the limit $\eps \to 0$, the fluid behaves as an incompressible one.
In the one-dimensional case, we expect the velocity to vanish due to the boundary conditions,
and the density to approach a constant because of the continuity equation.
In the presence of forcing, this necessarily imposes
the coefficient $\sigma_\eps$ to vanish as $\eps\to 0$: system \eqref{eq:NS} is therefore forced
by a collection of independent white noise processes, suitably colored in space
in an $\eps$-dependent fashion through the coefficient $\sigma_\eps=\eps^\beta \sigma$,
with $\beta>0$ and $\sigma$ fixed and independent of $\eps$.

The goal of this paper is to analyze the higher-order fluctuations in the vanishing Mach number limit.
For this purpose, we asymptotically expand the solutions
in the form\footnote{The first-order term of the density $\p$ vanishes by the momentum equation.}:
\begin{equation}\label{eq:devvar}
\begin{cases}
\rho = 1 + \eps^2 \rho_2 + \eps^3 \rho_3 + \cdots = 1+\eps^2 \bar{\rho}_\eps, \\[1mm]
u = \eps u_1+\eps^2 u_2+ \cdots = \eps \bar{u}_\eps.
\end{cases}
\end{equation}
We can now rewrite the system in the new variables $(\bpv, \buv)$ as the following \emph{fluctuation system}:
\begin{equation}\tag{FS}\label{eq:nonlinear}
			\begin{cases}
			\dis
			\d\bpv+\big(\eps (\bpv)_x \buv + \frac{1}{\eps}(1+\eps^2 \bpv) (\buv)_x \big)\d t = 0,\\[1.5mm]
			\dis\d\buv + \big(\eps \buv (\buv)_x + \frac{\gamma}{\eps} (1+\eps^2 \bpv)^{\gamma-2}(\bpv)_x \big)\d t
              = \frac{(\buv)_{xx}}{1+\eps^2 \bp} \d t + 	\eps^{\beta-1} \sigma \d W_t,
			\end{cases}
\end{equation}
with initial data:
\begin{equation}\label{ID-1}
\bp(0, x) = \bar{\p}_{0, \eps}(x), \quad \bu(0, x) = \bar{u}_{0,\eps}(x)  \qquad\text{ for all $x\in (0,1)$},
\end{equation}
and boundary conditions:
\begin{equation}\label{BC-1}
\bu(t, 0) = \bu(t, 1) = 0  \qquad\text{ for all $t >0$}.
\end{equation}
As $\eps \to 0$,  system \eqref{eq:nonlinear} is formally well-approximated by the following linear acoustic system:
\begin{equation}\tag{AS}\label{linear}
			\begin{cases}
				\dis\d r_\eps + \frac{1}{\eps}(v_\eps)_x \d t = 0,\\[2.5mm]
				\dis\d v_\eps + \frac{\gamma}{\eps}(r_\eps)_x \d t = (v_\eps)_{xx} \d t+\eps^{\beta-1} \sigma \d W_t,
			\end{cases}
\end{equation}
equipped with initial data:
\begin{equation}\label{ID-2}
r_\eps(0, x) = r_{0,\eps}(x), \quad v_\eps(0, x) = v_{0,\eps}(x)  \qquad\text{ for all $x\in (0,1)$},
\end{equation}
and boundary conditions:
\begin{equation}\label{BC-2}
v_\eps(t, 0) = v_\eps(t, 1) = 0  \qquad\text{ for all } t >0.
\end{equation}
The acoustic system \eqref{linear} is formally a stochastically forced linear damped wave equation:
\begin{equation*}
\d(r_\eps)_{t}-\frac{1}{\eps^2}(r_\eps)_{xx} \d t = (r_\eps)_{txx} \d t + \eps^{\beta-2}\sigma_x \d W_t.
\end{equation*}
The acoustic system \eqref{linear} is mathematically easier to study, compared to the original nonlinear system
\eqref{eq:nonlinear}, and it reveals the fine wave-like structure of the system in the vanishing Mach number limit.
 Justifying this limit is the main goal of this paper.

\subsection{Stochastic Navier-Stokes equations}
$\,$ The mathematical literature concerning the stochastic \\ Navier-Stokes equations is vast,
which has seen tremendous growth in the last decades.
We limit ourselves to recording a few relevant results concerning the existence and long-time behavior of solutions.
In the two-dimensional incompressible setting, the existence and uniqueness
of solutions have been established in various instances (different domains, boundary conditions,
regularity of initial data, and structure of the noise term);
see \cites{kuksin2012mathematics,da1996ergodicity}
and the references therein.
Moreover, the ergodicity and uniqueness of an invariant measure for the corresponding Markov semigroup
are also by now well-established under minimal requirements on the noise structure \cites{hairer2006ergodicity,flandoli1995ergodicity}.
The picture in the three-dimensional case is complicated by the well-known well-posedness issues.
The existence of global weak solutions is known (see \cite{Flandoli08} and the references therein),
and it has also been shown that there exists a strong Feller Markov selection with a unique invariant
measure \cites{DaPrato03, FlandoliRomito08}.

In comparison, the theory for compressible fluids is much less developed.
Most of the results concern the barotropic cases.
The existence of solutions has been obtained by using a
pathwise approach in \cite{1d_98} in the one-dimensional case,
in \cite{tornatore2000global} for a special two-dimensional case, and in \cite{feireisl2013compressible}
for the three-dimensional case.
Furthermore, the existence of global weak solutions with more general noise terms
and adiabatic exponent $\gamma>\frac{3}{2}$ in the three-dimensional setting has
been obtained in \cites{breit2016stochastic, smith2017random,wang2015global} independently
by using a several-layer approximation scheme, similar to the one used in the deterministic
case in \cite{feireisl01}.
In addition, the existence and uniqueness of local-in-time (up to a positive stopping time)
pathwise strong solutions have been obtained in  \cite{breit2018local}
for the barotropic Navier-Stokes equations with multiplicative noise on a torus.
Besides, the incompressible limit of the stochastic isotropic Navier-Stokes has been studied
in \cite{breit2016incompressible}.

There are only a few results available in the study of invariant measures and stationary solutions.
In \cite{gas2013mesure}, the existence of invariant measures has been proved for an approximated
version of the one-dimensional stochastic compressible Navier-Stokes equations in Lagrangian coordinates,
while \cite{michele} provides the construction of an invariant measure for system \eqref{eq:NS}
for the ideal gas pressure law, with $\gamma=1$.
In the three-dimensional case, stationary solutions to the isentropic Navier-Stokes equations
have been constructed in \cite{breit2019stationary}.
However, due to the uniqueness issues, stationarity is not intended in terms of an invariant measure
for a related Markov semigroup.  In this regard, a Markov selection
for the three-dimensional system has been constructed in \cite{markov_selection_3d_comp}.

\subsection{Vanishing Mach number limit}
Extensive research has been dedicated to the investigation of the low Mach number limit within the realm
of smooth solutions for the deterministic compressible Euler system,
even extending to multidimensional scenarios.
This rich body of work includes seminal contributions by Klainerman-Majda \cite{A1}, Majda \cite{A2},
Metivier-Schochet \cite{A3}, Schochet \cite{A4}, and a plethora of related references.

Chen-Christoforou-Zhang in \cite{A5} provided a rigorous mathematical foundation for the zero Mach limit,
specifically focusing on global discontinuous solutions, characterized by bounded variation,
of the compressible Euler equations. Their results illustrate the behavior of second-order coefficients
in the asymptotic expansion concerning the Mach number, highlighting the relevance of the linear acoustic
system within the Euler framework. In the context of the stochastic compressible Navier-Stokes equations,
the present study builds upon this framework, further developing and extending it.
Our work entails the derivation of the corresponding stochastic linear acoustic system,
as described in \eqref{linear}, and the subsequent validation of the vanishing Mach limit.

For additional insights into the low Mach number limit for global weak solutions pertaining
to the isentropic Navier-Stokes equations, we
refer to the work of Lions-Masmoudi \cite{VML5} for the deterministic setting
and to the work of Breit-Feireisl-Hofmanova \cite{breit2016incompressible} for the stochastic case.

\subsection{Main results}
We fix a filtered probability space $(\Omega, \Sigma, \mathbb{P}, (\mathcal{F}_t)_{t \geq 0})$
satisfying the usual conditions such that it
supports a sequence of independent real-valued Wiener processes $(W^k)_{k = 1 }^\infty$.
Let $\sigma = (\sigma_k)_{k=1}^\infty \subset H^1_0(0, 1) \cap H^2(0, 1)$ such that
\begin{equation} \label{eq:sigma_sum_cond}
\|\sigma_{xx}\|_{L^2}^2:=\sum_{k = 1}^\infty \intx |(\sigma_k)_{xx}|^2 \d x < \infty .
\end{equation}
This regularity assumption is tacitly assumed throughout this paper.
The noise term is defined as in the time-integrated sense as
\begin{align*}
\int_0^t \eps^{\beta} \p \sigma \d W_s :=  \sum_{k = 1}^\infty \int_0^t\eps^{\beta}\p \sigma_k \d W^k_s,
\end{align*}
where integrals are understood in the It\^{o}'s sense.

The basic energy functional for the compressible Navier-Stokes equations \eqref{eq:NS} can be written as
\begin{align} \label{eq:def_modified energy}
\cE(\p, u) =\intx \Big(\frac{1}{2}\p u^2 + H(\p) + \frac{1}{2}\frac{\p_x u}{\p} + \frac{1}{4} \frac{\p_x^2}{\p^3} \Big)\d x,
\end{align}
where $H(\p)$ is the pressure potential given by
\begin{align}\label{eq:def_H_p}
H(\p)=   \frac{1}{\eps^2}
\begin{cases}
\p\log \p - \p +1 &\text{ if } \gamma =1,\\[1mm]
 \frac{1}{\gamma-1}(\p^\gamma-\gamma(\p-1)) &\text{ if } \gamma >1.
\end{cases}
\end{align}
The energy structure revealed by $\cE$ is crucial in our analysis.
The first main result is a finite-time convergence of solutions of system \eqref{eq:nonlinear} to those of system \eqref{linear},
in the whole range of pressure exponents  $\gamma \geq 1$.

\begin{theorem} \label{thm:finite_time_convergence}
Let $\beta>\frac{1}{2}$ and $T>0$. Assume $(\bp_{0,\eps}, \bu_{0,\eps})_{\eps \in (0, 1]}$
and $(r_{0, \eps}, v_{0, \eps})_{\eps \in (0, 1]}$ are collections of $\mathcal{F}_0$-measurable
random variables in $H^1(0, 1) \times H^1_0(0, 1)$ such that
$$
\intx \bp_{0, \eps} \d x= 0, \quad 1+\eps^2 \bp_{0, \eps}>0,  \quad \intx r_{0, \eps} \d x= 0 \qquad\,\,\,\,\mbox{almost surely}.
$$
Furthermore, assume that there exist $\alpha>1$ and $\lambda>0$ such that
\begin{equation} \label{eq:assumption_moment_initial_poly}
\ex \left[ \cE(1+\eps^2 \bp_{0, \eps} , \eps \bu_{0, \eps})^p \right]
\lesssim  p! \, \lambda^p \eps^{\alpha p} \qquad \text{ for any } p\geq 1 \text{ and } \ \eps \in (0, 1].
\end{equation}
Then there exists $\eps_0 \in (0, 1]$  such that the corresponding solutions
$(\bp_{\eps}, \bu_{\eps})$ of system \eqref{eq:nonlinear}  and $(r_{\eps}, v_{\eps})$
of system \eqref{linear} satisfy the convergence estimate{\rm :}
\begin{equation}
 \ex \Big[\supt \big(\gamma\|\bpv-\rv\|_{L^2}^2+\|\buv-\vv\|_{L^2}^2\big)\Big]
 \lesssim \Big(\ex  \big[\gamma\|\bp_{0, \eps}-r_{0, \eps}\|_{L^2}^4 + \|\bu_{0, \eps}-v_{0, \eps}\|_{L^2}^4\big]\Big)^{\frac{1}{2}}
 +\eps^{\alpha-1} + \eps^{2\beta-1} \label{eq:theorem_1_L^2_diff_exp}
\end{equation}
for all $\eps\in (0,\eps_0]$.
\end{theorem}

We also translate the above result to the convergence of the invariant measures for $\gamma = 1$,
which is the second main result
in this paper.
The phase space for system \eqref{eq:nonlinear} is denoted by
\begin{equation*}
\bXv = \Big\{ (\bpv, \buv) \in  H^1(0, 1) \times H^1_0(0, 1)\,: \,\,
\intx \bpv \d x = 0, \,\, 1+\eps^2 \bpv >0 \Big\}.
\end{equation*}
Thanks to the results in \cite{michele}, the corresponding Markov semigroup has
an invariant probability measure $\bar{\mu}_\eps$ on $\bXv$;
see Section \ref{section:convergence_invariant_measures}.
On the other hand, the Markovian framework for system \eqref{linear} can be set up in the larger phase space:
\begin{align} \label{eq:state_space_acoustic}
\bX = \Big\{ (r, v) \in  H^1(0, 1) \times H^1_0(0, 1)\, :\, \intx r \d x = 0\Big\}.
\end{align}
We denote by $\omega_\eps$ the unique invariant probability measure of system \eqref{linear} in $\bX$.

\begin{theorem}\label{thm:convergence_invariant_measures}
Let $\gamma = 1$. There exists $\eps_0>0$ such that, for all $0 < \eps < \eps_0$,
\begin{align*}
\mathcal{W}_{L^2} (\bar{\mu}_\eps, \omega_\eps) \lesssim \eps^{\beta-\frac{1}{2}},
\end{align*}
where $\mathcal{W}_{L^2}$ is the Wasserstein metric on $\bX$ defined
by using the $L^2 \times L^2$ norm\footnote{We have abused the notation $\bar{\mu}_\eps$
to mean the Borel probability measure on $\bX$ that naturally extends $\bar{\mu}_\eps$.
See \eqref{eq:measure_extension} for the precise definition.}.
\end{theorem}

\begin{remark} Concerning our main results, we have several remarks in order{\rm :}
\begin{enumerate}
\item[(i)] Assumption \eqref{eq:sigma_sum_cond} implies in particular that
\begin{equation*}
\|\sigma\|_{L^\infty}^2:=  \sum_{k = 1}^\infty \|\sigma_k\|_{L^\infty}^2 \leq  \sum_{k = 1}^\infty \|(\sigma_k)_x\|_{L^2}^2 < \infty.
\end{equation*}
This
covers the natural choice that
$\sigma_k = \alpha_k \sin (k \pi x)$ with $(\alpha_k)_{k=1}^\infty \subset \mathbb{R}$ and $\sum_{k=1}^\infty k^4 \alpha_k^2 < \infty$.

\item[(ii)] The results apply to the case: $\beta\in (\frac{1}{2}, 1)$,
where the noise term in the fluctuation system \eqref{eq:nonlinear} blows up as $\eps \to 0$.

\item[(iii)] The results above are better than the corresponding situation with a deterministic external forcing,
in which case $\beta>1$ is required, namely the external forcing must vanish with the Mach number.
The improvement in our cases is due to the martingale structure of the stochastic integrals.
\end{enumerate}
\end{remark}

\subsection{Techniques and key points of the analysis}
Before presenting the detailed analysis in the subsequent sections,
we summarize here the novel ingredients and key observations that constitute the foundation of our approach.

\diampar{Polynomial and exponential moment estimates for the energy for all $\gamma \geq 1$}
These generalize the results obtained in \cite{michele} where only the isothermal pressure law ($\gamma =1$)
is considered (see Proposition \ref{prop:poly_and_exp_moments_bounds_NS}).
Instead of trying to control the source terms from the random forcing
by using the dissipative terms (which are degenerate for $\gamma >1$)
in the energy balance (see \eqref{eq:modified_energy_balance} below) as in \cite{michele},
we directly use $\cE$ to estimate them, resulting in
the bounds that grow exponentially in time.

\diampar{Polynomial moment estimates on the upper and lower bounds of the density at small Mach numbers}
These estimates are a consequence of the exponential moment
estimates for $\cE$. The estimates allow us to obtain fundamental controls on the fluctuation variables,
which play a critical role on the finite-time convergence analysis
(see Proposition \ref{prop:expectation_non_linear}).

\diampar{Exponential contraction property of the acoustic system due to the viscosity}
We prove that the acoustic system has a fine structure of global exponential contraction.
This property allows the analysis of the large-time behavior of the system and is of paramount importance
in translating the finite-time convergence analysis to the study of the convergence of
invariant measures (which could informally be considered as capturing the information
on the behavior of the system when $t = \infty$). There is a mismatch between the topology
used to prove the finite-time convergence and the contraction property:
In order to overcome this issue,
we employ an $\eps$-weighted $H^1$ norm that is stronger than the $L^2$ norm,
but coincides with the $L^2$ norm in the limit $\eps\to 0$.

\diampar{Novel estimates of the invariant measures}
We prove that statistically stationary solutions distributed as $\bar{\mu}_\eps$
satisfy the assumption of Theorem \ref{thm:finite_time_convergence}.
These are some novel estimates on the invariant measures of the compressible Navier-Stokes equations.
In particular, we obtain the finite exponential moments of $\cE$.
We believe these estimates are of independent interest.

\subsection{Organization of the paper}
In Section \ref{section:modified_energy_estimates_NS},
we make the estimates of solutions of the Navier-Stokes equations \eqref{eq:NS}
and use them as a starting point for the analysis of system \eqref{eq:nonlinear} in Section \ref{section:estimates_pertutrbation_system}.
Next, we study the properties of the acoustic system \eqref{linear} in Section \ref{section:acoustic}.
All these estimates are gathered together in Section \ref{section:finite_time_analysis}
to give a complete proof of Theorem \ref{thm:finite_time_convergence}.
The study of the properties and the convergence of invariant measures is left for Section \ref{section:convergence_invariant_measures}.

Throughout the paper, we adopt the convention to denote $C$ as the constants whose exact values are irrelevant
and may be different at each occurrence.
We also use the subscripts to emphasize the dependence of the constants on various parameters.
In all cases, $C$ is
independent of $\varepsilon$.

\section{Energy Structure for the Navier-Stokes Equations} \label{section:modified_energy_estimates_NS}
In this section, we study the properties of solutions of problem \eqref{BC-0}--\eqref{ID-0} for system \eqref{eq:NS} in detail.
Considering the pathwise transformation
$$
\huv = \uv -  \sum_{k = 1}^\infty \int_0^t\eps^{\beta}\sigma_k \d W^k_s,
$$
we can transform system \eqref{eq:NS} into a system with random coefficients but without stochastic integrals.
Then, using the ideas in \cite{onedns}, we can obtain the existence and uniqueness
of strong pathwise solutions of problem \eqref{BC-0}--\eqref{ID-0}
for system \eqref{eq:NS} (see also \cites{1d_96, 1d_98, Hoff00, michele}):

\begin{proposition} \label{prop:existence_N-S}
Assume $(\p_{0, \eps}, u_{0, \eps})_{\eps \in (0, 1]}$ is a collection of $\mathcal{F}_0$-measurable
random variables in $H^1\times H^1_0$ such that
$$
\intx \p_{0, \eps} \d x  =1, \quad \p_{0, \eps}>0 \qquad\,\,\mbox{almost surely}.
$$
Then there exists a unique adapted strong pathwise solution $(\pv, \uv)$ of system \eqref{eq:NS} such that
\begin{equation*}
\pv \in C([0, \infty) \times [0, 1]) \cap L^\infty_{\mathrm{loc}}(0, \infty; H^1),
\,\,\,\, \uv \in C([0, \infty); H^1_0) \cap L^2_{\mathrm{loc}}(0, \infty; H^2)
\qquad\, \mbox{almost surely}.
\end{equation*}
Moreover, the map $t \mapsto \p(t, \cdot) \in H^1$ is weakly continuous
and, for every $t>0$, the density is bounded away from vacuum{\rm :}
\begin{equation*}
\p(t, \cdot)^{-1} \in L^\infty(0, 1)\qquad\, \mbox{almost surely}.
\end{equation*}

\end{proposition}
We skip the proof of this classical result and limit ourselves to derive  some \emph{a priori} estimates for the system.

\subsection{The main energy functional and its properties}
The classical {\it a priori} estimates of solutions of the compressible Navier-Stokes equations
usually begin with the physical energy. However, in the one-dimensional case, we can use
the energy $\cE$ defined in \eqref{eq:def_modified energy},
which is more powerful as it provides additional estimates
on the $H^1$-norm of the density.
A similar structure in Lagrangian coordinates appeared in \cite{Kanel}.
The energy has been used in \cites{Hoff99, Hoff00} for deterministic isothermal flows to study
the large-time behavior of solutions. Its generalization
for the density-dependent viscosities, which even applies to some multidimensional cases,
is sometimes known as the BD-entropy that was introduced in
Bresch-Desjardins \cites{bresch2003existence, bresch2004some, bresch2007existence, bresch2003some}.
The entropy has then been used critically to prove
various {\it a priori} estimates, particularly on the lower bound of the density
in the one-dimensional case (see e.g. \cites{onedns, li2008vanishing}).
We now record some properties of $\cE(\p, u)$.

\begin{proposition} \label{prop:property_modified_energy}
Let $(\p, u) \in H^1 \times H^1_0$ such that $\intx \p(x) \d x = 1$ and $\p(x) >0$. Then
\begin{enumerate}
\item[\rm (i)] $\dis \intx \p u^2 \d x \leq 4 \cE(\p, u)$,
\vspace{1mm}
\item[\rm (ii)] $\dis \intx \frac{|\p_x|^2}{\p^3} \d x \leq 8 \cE(\p, u)$,\vspace{1mm}
\item[\rm (iii)] $\dis \intx H(\p) \d x\leq \cE(\p, u)$,\vspace{1mm}
\item[\rm (iv)] $\dis \|\p^{-1}\|_{L^\infty} \leq 2\big(1+2\cE(\p, u)\big)$,\vspace{1mm}
\item[\rm (v)] $\dis \max\{\|\p^{-1}\|_{L^\infty}, \|\p\|_{L^\infty}\}
\leq \exp (\sqrt{8\cE(\p, u)})$.
\end{enumerate}
\end{proposition}

\begin{proof}
The first two estimates follow respectively from the pointwise inequalities
\begin{equation*}
\left|\frac{\p_x u}{\p}\right|
\leq \frac{1}{2}\p u^2 + \frac{1}{2}\frac{|\p_x|^2}{\p^3},
\qquad
\left|\frac{\p_x u}{\p}\right|
\leq \p u^2 + \frac{1}{4}\frac{|\p_x|^2}{\p^3}.
\end{equation*}
The third estimate follows from the fact that
$$
\frac{1}{2}\p u^2 + \frac{1}{2}\frac{\p_x u}{\p} + \frac{1}{4} \frac{|\p_x|^2}{\p^3} \geq 0,
$$
which can be seen from the above two inequalities.

To prove the last two estimates, we note that
there exists $x_0 \in [0, 1]$ such that $\p(x_0)=1$, since $\intx \p(x) \d x =1$ and $\rho(x)$ is continuous.
Hence, for all $x \in [0, 1]$,
\begin{align*}
\big|\p^{-\frac{1}{2}}(x)-1\big|&= \big|\p^{-\frac{1}{2}}(x)- \p^{-\frac{1}{2}}(x_0)\big|
= \Big|\int_{x_0}^x \big(\p^{-\frac{1}{2}}(y)\big)_y \d y \Big| \\
&\leq \frac{1}{2}\intx \big|\rho^{-\frac{3}{2}}\p_x\big| \d x
\leq \frac{1}{2} \Big(\intx \frac{|\p_x|^2}{\p^3} \d x\Big)^{\frac{1}{2}} \leq  \sqrt{2\cE(\p, u)}.
\end{align*}
The triangle inequality then implies the desired estimate. Likewise, we have
\begin{align*}
|\log \p(x)| &= | \log \p(x)-\log\p(x_0) |
= \Big|\int_{x_0}^x \big(\log \p(y)\big)_y \d y \Big|
\leq \Big(\intx \p(y)\d y\Big)^{\frac{1}{2}}\Big(\intx \frac{\p_x^2}{\p^3}\d y\Big)^{\frac{1}{2}}.
\end{align*}
The first integral in the inequality above is $1$ due to the mass condition,
while the second integral is bounded by $8 \cE (\p, u)$.
Exponentiating both sides gives the last estimate.
\end{proof}

\subsection{Energy estimates for the Navier-Stokes equations}\label{sec:2.2}
For solutions of system \eqref{eq:NS}, we can now derive the following estimates:

\begin{proposition} \label{prop:poly_and_exp_moments_bounds_NS}
Assume $(\p_{0, \eps}, u_{0, \eps})_{\eps \in (0, 1]}$ is a collection of $\mathcal{F}_0$-measurable
random variables in $H^1\times H^1_0$ such that
$$
\intx \p_{0, \eps} \d x  =1, \quad\, \p_{0, \eps}>0 \qquad \mbox{almost surely}.
$$
Then the corresponding solution $(\pv, \uv)$ of system \eqref{eq:NS} satisfies the following properties{\rm :}

\begin{enumerate}
\item[\rm (i)] For all $p \geq 2$, there exists $C_p>0$ such that, for all $\eps \in (0, 1]$ and $T \geq 1$,
\begin{align}
&\ex \Big[\sup_{0\leq t \leq T} \Big(\cE(\pv, \uv) + \frac{1}{2} \int_0^t \|(\uv)_x\|_{L^2}^2 \d s
+ \frac{1}{2\eps^2 } \int_0^t \|\pv^{\frac{\gamma-3}{2}}(\pv)_x\|_{L^2}^2 \d s\Big)^p\Big] \nonumber\\[1mm]
&\,\leq  3^{p-1}\e^T \, \ex \big[\cE(\p_{0, \eps},u_{0, \eps})^p\big]
+ C_p \, \e^T \, T^{p+1} \eps^{2\beta p} \|\sigma\|^{2p}_{L^\infty},\label{eq:orig_poly_bound}
\end{align}
with $C_p = O(1296^p p^p)$ as $p \to \infty$.

\item[\rm (ii)] For
\begin{equation}\label{def:eq_eta_0}
\eta_0(\eps, T) := (4000 T \eps^{2\beta} \|\sigma\|^{2}_{L^\infty})^{-1},
\end{equation}
there exists $C>0$ such that, for all $\eps \in (0, 1]$,  $T \geq 1$, and $0 < \eta < \eta_0(\eps, T)$,
\begin{align}\label{eq:orig_expo_bound}
&\ex \Big[\exp \Big(\eta\sup_{0\leq t \leq T} \big(\cE(\pv, \uv) + \frac{1}{2} \int_0^t \|(\uv)_x\|_{L^2}^2 \d t
+ \frac{1}{2\eps^2 } \int_0^t \|\pv^{(\gamma-3)/2}(\pv)_x\|_{L^2}^2 \d t  \big)\Big) \Big]\nonumber\\[1mm]
&\leq  \frac{\e^T}{3}
\ex \big[\exp\big(3 \eta \cE(\p_{0, \eps},u_{0, \eps})\big)\big]
+\frac{C T\e^T}{1-\eta / \eta_0(\eps, T)}.
\end{align}
\end{enumerate}
\end{proposition}

\begin{proof}
For the first estimate, we employ the It\^{o}'s formula to perform a direct computation
to obtain the energy balance:
\begin{align}\label{eq:modified_energy_balance}
&\d \cE(\p_\eps, u_\eps)
+\frac{1}{2} \Big(\intx |(\uv)_x|^2 \d x  + \frac{1}{\eps}\intx \pv^{\gamma-3}|(\pv)_x|^2 \d x\Big) \d t \nonumber\\
&= \frac{1}{2}\eps^{2\beta}\intx\pv \sigma^2 \d x \d t + \eps^\beta \intx \Big(\pv \uv + \frac{1}{2} \frac{(\pv)_x}{\pv}\Big) \sigma \d x \d W_t.
\end{align}
Note that the last two terms are interpreted as
\begin{align*}
& \intx\pv \sigma^2 \d x \d t = \sum_{k=1}^\infty \intx \pv \sigma_k^2 \d x \d t,\\
& \intx \Big(\pv \uv + \frac{1}{2} \frac{(\pv)_x}{\pv}\Big) \sigma \d x \d W_t
  =  \sum_{k=1}^\infty \intx \Big(\pv \uv + \frac{1}{2} \frac{(\pv)_x}{\pv}\Big) \sigma_k(x) \d x \d W^k_t;
\end{align*}
also see \cite{onedns} for the similar computation in the deterministic case without the It\^{o}'s correction
term and \cite{michele} for the computation in the stochastic case with $\gamma=1$.
Therefore, for all $p \in [2, \infty)$ and $T \geq 1$, it holds almost surely that
\begin{align*}
&\bigg(\sup_{0\leq t \leq T} \Big(\cE(\p_\eps, u_\eps)
+\frac{1}{2} \intt \intx |(\uv)_x|^2 \d x \d s
+ \frac{1}{2\eps^2}\intt \intx \pv^{\gamma-3}|(\pv)_x|^2 \d x \d s \Big) \bigg)^p\\
&=\bigg(\sup_{0\leq t \leq T} \Big( \cE(\p_{0, \eps}, u_{0, \eps})
   +  \frac{\eps^{2\beta}}{2}\intt\intx\pv \sigma^2 \d x \d s
+ \eps^\beta \intt \intx \big(\pv \uv + \frac{1}{2} \frac{(\pv)_x}{\pv}\big) \sigma \d x \d W_s \Big) \bigg)^p\\
&\leq 3^{p-1} \bigg(\cE(\p_{0, \eps}, u_{0, \eps})^p+\sup_{0\leq t \leq T} \Big(\frac{\eps^{2\beta}}{2}\intt\intx\pv \sigma^2 \d x \d s\Big)^p
\\&
\qquad \qquad
+ \sup_{0\leq t \leq T} \Big(\eps^\beta \intt \intx \big(\pv \uv + \frac{1}{2} \frac{(\pv)_x}{\pv}\big) \sigma \d x \d W_s \Big)^p \bigg).
\end{align*}
We now estimate the last two terms in the bracket on the right-hand side. For the first term, we have
\begin{align*}
\ex \Big[ \sup_{0\leq t \leq T} \Big(\frac{\eps^{2\beta}}{2}\intT\intx\pv \sigma^2 \d x \d t\Big)^p\Big]
&= \frac{\eps^{2{p\beta}}}{2^p}\,\sup_{0\leq t \leq T}
\Big(\sum_{k=1}^\infty \intT\intx\pv \sigma_k^2 \d x \d t\Big)^p\\
&\leq \frac{\eps^{2{p\beta}}}{2^p}\,
\sup_{0\leq t \leq T}  \Big(\sum_{k=1}^\infty \|\sigma_k\|_{L^\infty}^2 \intT\intx\pv  \d x \d t\Big)^p\\
&\leq \frac{\eps^{2{p\beta}}}{2^p}\,\sum_{k=1}^\infty \|\sigma_k\|_{L^\infty}^{2p} T^p
\leq \frac{\eps^{2{p\beta}}}{2^p} \|\sigma\|_{L^\infty}^{2p} T^p.
\end{align*}
Next, we estimate the second term involving a (local) martingale.
For a local martingale $(N_t)_{t \geq 0}$, we write $(\langle N \rangle_t)_{t \geq 0}$
for its quadratic variation process.
Now, by the Burkholder-Davis-Gundy inequality, there exists $G_p >0$ such that
\begin{align*}
&\ex \bigg[ \sup_{0\leq t \leq T} \Big( \eps^\beta\intt \intx
\big(\pv \uv + \frac{1}{2} \frac{(\pv)_x}{\pv}\big) \sigma \d x \d W_s \Big)^p\bigg]\\
&=\ex \bigg[ \sup_{0\leq t \leq T} \Big( \sum_{k=1}^\infty \eps^\beta
\intt \intx \big(\pv \uv + \frac{1}{2} \frac{(\pv)_x}{\pv}\big) \sigma_k \d x \d W^k_s \Big)^p\bigg]\\
&\leq G_p \, \eps^{p \beta} \, \ex \bigg[\supt \Big\langle\sum_{k=1}^\infty \intt \intx
\big(\pv \uv + \frac{1}{2} \frac{(\pv)_x}{\pv}\big) \sigma_k \d x \d W^k_s \Big\rangle^{\frac{p}{2}}\bigg]\\
&= G_p \, \eps^{p \beta} \, \ex \bigg[\Big(\sum_{k=1}^\infty
\intT \Big(\intx \big(\pv \uv + \frac{1}{2} \frac{(\pv)_x}{\pv}\big) \sigma_k \d x \Big)^2 \d s \Big)^{\frac{p}{2}}\bigg]\\
&\leq G_p \,  \eps^{p \beta} \, \|\sigma\|_{L^\infty}^p \, \ex \bigg[\Big(\intT \intx\pv \uv^2  \d x \d t
  + \intT \intx \frac{|(\pv)_x|^2}{\pv^3} \d x \Big)^{\frac{p}{2}}\bigg]\\
&\leq 12 G_p \,  \eps^{p \beta} \, \|\sigma\|_{L^\infty}^p \, \ex \bigg[ \Big(\intT \cE(\p_\eps, u_\eps)\d t \Big)^{\frac{p}{2}}\bigg]\\
&\leq 12 G_p \,  \eps^{p \beta} \, T^{\frac{p}{2}-1} \|\sigma\|_{L^\infty}^p \, \ex \Big[\intT \cE(\p_\eps, u_\eps)^{\frac{p}{2}}\d t \Big]\\
&\leq 3^{-(p-1)} \ex \Big[ \intT \cE(\p_\eps, u_\eps)^{p}\d t \Big]
+ \frac{144^p}{4} 3^{p-1} G_p^2  T^{p-1} \eps^{2p\beta} \|\sigma\|_{L^\infty}^{2p}.
\end{align*}
As a result, we have
\begin{align*}
&\ex\bigg[\sup_{0\leq t \leq T} \Big( \cE(\p_\eps, u_\eps)
+\frac{1}{2} \intt \intx |(\uv)_x|^2 \d x \d s
+ \frac{1}{2\eps^2}\intt \intx \pv^{\frac{\gamma-3}{2}}|(\pv)_x|^2 \d x \d s \Big)^p \bigg]\\
&\leq \ex \Big[ \intT \cE(\p_\eps, u_\eps)^{p}\d t \Big] + 3^{p-1} \ex \Big[\cE(\p_{0, \eps}, u_{0, \eps})^p\Big]
+ \Big(\frac{1296^p}{36}  G_p^2   T^{p-1}
+ \frac{3^{p-1}}{2^p} T\Big)T^{p-1} \eps^{2\beta p} \|\sigma\|_{L^\infty}^{2p}.
\end{align*}
The assertion now follows from the Gronwall inequality.

In order to
show the asymptotic estimate of the constant in \eqref{eq:orig_poly_bound}, we first observe that
\begin{equation*}
N_t := \eps^\beta\intt \intx \big(\pv \uv + \frac{1}{2} \frac{(\pv)_x}{\pv}\big) \sigma \d x \d W_s
\end{equation*}
is a continuous martingale. To see this, we compute similar to the above that,
for all $t \geq \tau \geq 0$ and $m \geq 2 $,
\begin{align*}
\ex \big[|N_t-N_\tau|^m \big]
&= \eps^{\beta m} \ex \bigg[ \Big( \sum_{k=1}^\infty \int_\tau^t \intx \big(\pv \uv + \frac{1}{2} \frac{(\pv)_x}{\pv}\big) \sigma_k \d x \d W^k_s \Big)^m\bigg]\\
&\leq G_m \, \eps^{\beta m} \,
\ex \bigg[\Big(\sum_{k=1}^\infty \int_\tau^t \intx \big(\pv \uv + \frac{1}{2} \frac{(\pv)_x}{\pv}\big) \sigma_k \d x \d W^k_s \bigg)^{\frac{m}{2}}\Bigg]\\
&\leq 12 G_m \,  \eps^{\beta m} \, \|\sigma\|_{L^\infty}^m \, \ex \bigg[ \Big(\int_\tau^t \cE(\p_\eps, u_\eps)\d r \Big)^{\frac{m}{2}}\bigg]\\
&\leq 12 G_m \,  \eps^{\beta m} \, \|\sigma\|_{L^\infty}^m \, (t-\tau)^{\frac{m}{2}}
\ex \bigg[ \Big(\sup_{ s \leq r \leq t} \cE(\p_\eps(t, \cdot), u_\eps(t, \cdot)) \Big)^{\frac{m}{2}}\bigg]\\
&= C_m \eps^{\beta m} (t-\tau)^{\frac{m}{2}} .
\end{align*}
Since $m\geq 2$ can be arbitrary large, by the Kolmogorov continuity theorem,
$N_t$ is almost surely H{\"o}lder continuous with H{\"o}lder exponent as any number less than $\frac{1}{2}$.
In particular, since $N_t$ is continuous, the asymptotic estimate of the constant in \eqref{eq:orig_poly_bound}
follows from the fact that $G_p = O(p^{\frac{p}{2}})$  as $p \to \infty$, thanks to \cite{BDG}.

We now show \eqref{eq:orig_expo_bound}.
Define
\begin{equation*}
Y_t = \sup_{0\leq t \leq T} \Big(\cE(\pv, \uv) + \frac{1}{2} \int_0^t \|(\uv)_x\|_{L^2}^2 \d s
+ \frac{1}{2\eps^2 } \int_0^t \|\pv^{\frac{\gamma-3}{2}}(\pv)_x\|_{L^2}^2 \d s\Big).
\end{equation*}
Then
\begin{align*}
\exp (\eta Y_t)
&= \sum_{k = 0}^\infty \frac{\eta^k }{k!}\ex \big[Y_t^k\big]\\
&\leq \sum_{k = 0}^\infty \frac{\eta^k }{k!} \Big(3^{k-1}  \e^t \,
   \ex \big[\cE(\p_{0, \eps},u_{0, \eps})^k\big]
   + C_k  \e^t t^{k+1} \eps^{2\beta k} \|\sigma\|^{2k}_{L^\infty}\Big)\\
&\leq  \frac{\e^t}{3} \sum_{k = 0}^\infty \frac{(3\eta)^k }{k!}
  \ex \big[\cE(\p_{0, \eps},u_{0, \eps})^k\big]
   + C t\e^t \sum_{k = 0}^\infty \frac{\big(1296t \eta \eps^{2\beta} \|\sigma\|^{2}_{L^\infty}\big)^k }{k!} k^k\\
&\leq  \frac{\e^t}{3} \ex \big[\exp\big(3 \eta \cE(\p_{0, \eps},u_{0, \eps})\big)\big]
 +  C t\e^t \sum_{k = 0}^\infty \big(4000t \eta  \eps^{2\beta} \|\sigma\|^{2}_{L^\infty}\big)^k\\
&\leq  \frac{\e^t}{3} \ex \big[\exp\big(3 \eta \cE(\p_{0, \eps},u_{0, \eps})\big)\big]
  + \frac{C t\e^t}{1-4000t \eta  \eps^{2\beta} \|\sigma\|^{2}_{L^\infty}},
\end{align*}
if $0 < \eta < \min\{\frac{\eta(\eps)}{3},\eta_0(\eps, T)\}$.
Note that we have used the Stirling's approximation $k^k \sim k!\, \e^k (2\pi k)^{-\frac{1}{2}}$
for large $k$ in the penultimate inequality.
This completes the proof.
\end{proof}

\begin{remark}
In the isothermal case $\gamma =1$, the estimates in {\rm Proposition \ref{prop:poly_and_exp_moments_bounds_NS}}
can be improved. In particular, the factor $\e^T$ in the first estimate can be replaced by a $T$-independent
constant and the smallness of $\eta$ needed in the second estimate can also be chosen to be independent of  $T$;
see \cite{michele}.
\end{remark}

\subsection{Estimates for the fluctuation system}\label{section:estimates_pertutrbation_system}
The energy estimates established for system \eqref{eq:NS} in \S \ref{sec:2.2}
constitute the basis for the analysis
of the fluctuation system \eqref{eq:nonlinear}.
The estimates here will be critical to establish
the finite-time convergence results in Theorem \ref{thm:finite_time_convergence}.
We first restate Proposition \ref{prop:existence_N-S} in the new variables $(\bpv, \buv)$ defined in \eqref{eq:devvar}.

\begin{proposition} \label{prop:existence_nonlinear}
Assume that $(\bp_{0, \eps}, \bu_{0, \eps})_{\eps \in (0, 1]}$ is a collection of
$\mathcal{F}_0$-measurable random variables in $H^1\times H^1_0$ such that
$$
\intx \bp_{0, \eps} \d x= 0, \quad 1+\eps^2 \bp_{0, \eps}>0 \qquad\,\,\,  \mbox{almost surely}.
$$
Then there exists a unique adapted strong pathwise solution $(\bpv, \buv)$ to \eqref{eq:nonlinear} such that
\begin{equation*}
\bpv \in C([0, \infty) \times [0, 1]) \cap L^\infty_{\mathrm{loc}}(0, \infty; H^1), \,\,\,\,
\buv \in C([0, \infty); H^1_0) \cap L^2_{\mathrm{loc}}(0, \infty; H^2) \qquad\mbox{almost surely}.
\end{equation*}
Moreover, the map: $t \mapsto \bpv(t, \cdot) \in H^1$ is weakly continuous and, for each $t >0$,
\begin{equation*}
(1+\eps^2\bpv(t, \cdot))^{-1} \in L^\infty(0, 1) \qquad\,\, \mbox{almost surely}.
\end{equation*}
\end{proposition}

We now further assume that \eqref{eq:assumption_moment_initial_poly} holds.
In particular, this implies that there exists $C>0$ such that, for all $p \geq 1$ and $\eps \in (0, 1]$,
for all $0 \leq \eta < \lambda^{-1} \eps^{-\alpha}$,
\begin{align} \label{eq:exponential_moment_initial}
\ex \left[\exp\big(\eta \mathcal{E}(1+\eps^2 \bp_{0, \eps} , \eps \bu_{0, \eps})\big) \right]
\leq \frac{C}{1-\eta \lambda \eps^\alpha}   \qquad \mbox{for any $p\geq 1, \eps \in (0, 1],
 \eta \in(0,\lambda^{-1} \eps^{-\alpha})$}.
\end{align}
In fact, the assumption is equivalent to asserting that
$\mathcal{E}(1+\eps^2 \bp_{0, \eps} , \eps \bu_{0, \eps})$ has an exponential tail.
That is, there exist $\tilde{C}, \tilde{\lambda}>0$ such that, for all $R >0$,
\begin{equation*}
\mathbb{P}\left[ \mathcal{E}(1+\eps^2 \bp_{0, \eps} , \eps \bu_{0, \eps})\geq R  \right]
\leq \tilde{C} \e^{-\frac{R}{\tilde{\lambda}\eps^\alpha}}.
\end{equation*}

\begin{remark}
The assumption on the initial data covers some important cases of interest.
For example, if the initial data are deterministic and bounded
in $H^1 \times H^1_0$ uniformly in $\eps \in (0, 1]$,
then \eqref{eq:assumption_moment_initial_poly} is satisfied with $\alpha =2$.
Also, as we will see later, for $\gamma = 1$,
if $(\bp_{0, \eps}, \bu_{0, \eps})$ is a statistically stationary solution of
the fluctuation system \eqref{eq:nonlinear}, then the assumption is satisfied with $\alpha = 2 \beta$.
\end{remark}

We now quantify better the estimates of solutions of system \eqref{eq:nonlinear}.

\begin{proposition} \label{prop:expectation_non_linear}
Let $T \geq 1$, and let $(\bp_{0, \eps}, \bu_{0, \eps})$ be
satisfying \eqref{eq:assumption_moment_initial_poly} as above. Then

\begin{enumerate}
\item[\rm (i)]\label{item:densbd}
For all $p\geq 1$, there exists $C>0$, independent of $(\eps, \alpha, \beta)$, such that
\begin{align} \label{eq:density_bounds_infty_expectation}
&\ex \Big[\max \big\{\supt \|(1+\eps^2 \bpv)^p\|_{L^\infty}, \supt \|(1+\eps^2 \bpv)^{-p}\|_{L^\infty}\big\}\Big] \leq C
\end{align}
for all $\eps\in (0, \eps_0]$ for some $\eps_0>0$ depending only on $p, T, \alpha, \beta$, and $\|\sigma\|_{L^\infty}$.

\smallskip
\item[\rm (ii)]\label{item:solbd}
For all $p \geq 1$, there exists $C_p = C_p (p, \gamma, \lambda, T)>0$, independent of $(\alpha, \beta, \eps)$,
such that, for all $\eps\in (0, \eps_0]$ $($ for $\eps_0>0$ depending on $p, T, \alpha, \beta$,
and $\|\sigma\|_{L^\infty}${\rm )},
\begin{align}
& \ex \bigg[ \Big(\supt\left( \|\bpv\|_{L^2}^2+\|\buv\|_{L^2}^2+\eps^2 \|(\bpv)_x\|_{L^2}^2\right) \d x
+\intT \big(\|(\bpv)_x \|^2_{L^2} + \|(\buv)_x\|_{L^2}^2\big)  \d t\Big)^p\bigg] \nonumber\\[1mm]
&\leq C_p \big(\eps^{(\alpha-2)p} + \eps^{(2\beta-2)p} \big), \label{eq:enerone}
\end{align}

\smallskip
\item[\rm (iii)]\label{item:expbd} There exists $C= C(T)$ such that, for all $\eps\in (0, 1]$ and $\delta>0$ with
\begin{align} \label{eq:expectation_exp_u_delta}
\delta<\min\big\{\frac{1}{6\lambda \eps^{\alpha-2}},\,\frac{\eta_0(\eps, T)\eps^2}{2}\big\}
\end{align}
and $\eta_0$ defined in {\rm Proposition \ref{prop:poly_and_exp_moments_bounds_NS}},
\begin{equation}\label{2.8a}
\ex \Big[\exp \Big( \delta\intT \intx \buv^2 \d x  \d t\Big) \Big]
\leq C.
\end{equation}
\end{enumerate}
\end{proposition}

\begin{proof} We divide the proof into three steps.

\smallskip
1. We first go back to the original variables $\pv = 1+\eps^2 \bpv$ and $\uv = \eps \buv$, which solve \eqref{eq:NS}
with the initial data $(1+\eps^2 \bp_{0, \eps}, \eps \bu_{0, \eps})$.

Fix $p \geq 1$. To prove \eqref{eq:density_bounds_infty_expectation},
since $\alpha, \beta >0 $, there exists $\eps_0 >0$ such that, for all $\eps \in (0, \eps_0)$,
$8p^2 \leq \min\{(6\lambda \eps^\alpha)^{-1} , \frac{\eta_0(\eps, T)}{2}\}$, where $\eta_0(\eps, T)$ is defined in \eqref{def:eq_eta_0}.
For all such $\eps$, by Proposition \ref{prop:property_modified_energy},
we have
\begin{align*}
\max\Big\{\supt \|\pv^p\|_{L^\infty},\, \supt \|\pv^{-p}\|_{L^\infty}\Big\}
\leq \exp \Big(\big(\supt 8p^2 \cE(\pv, \uv)\big)^{\frac{1}{2}}\Big).
\end{align*}
We note that the function $\mu : [1, \infty) \to [1, \infty)$, defined
by $\mu (x) = \exp((\log x)^{\frac{1}{2}})$ for all $x \geq 1$, is concave and strictly increasing.
Therefore, we have
\begin{align*}
&\max \Big\{\supt \|\pv^p\|_{L^\infty}, \,\supt \|\tp^{-p}\|_{L^\infty}\Big\} \\
&\leq \ex\Big[\exp \Big(\big(\supt 8p^2 \cE(\pv, \uv)\big)^{\frac{1}{2}}\Big)\Big]\\
&= \ex \big[\mu( \exp (\supt 8p^2 \cE(\pv, \uv)))\big]\\
&\leq \mu(\ex [ \exp (\supt 8p^2\cE(\pv, \uv))])\\
&\leq \mu(\frac{\e^T}{3} \ex [\exp(\frac{3}{6\lambda \eps^\alpha} \cE(\p_{0, \eps},u_{0, \eps}))]
+\frac{C T\e^T}{1-(\eta_0(\eps, T)/2) / \eta_0(\eps, T)})\\
&= \mu(\frac{\e^T}{3} \ex [\exp(\frac{1}{2\lambda \eps^\alpha} \cE(\p_{0, \eps},u_{0, \eps}))]+2C T\e^T)\leq C,
\end{align*}
by means of \eqref{eq:orig_expo_bound} and \eqref{eq:exponential_moment_initial}.

\smallskip
2. We now tackle \eqref{eq:enerone} and start
by estimating  $\ex[\|\bpv\|_{L^\infty(0, T; L^2)}^{2p}]$.
By Taylor's theorem, for all $z \in \mathbb{R}$, there exists $\xi$ between $1$ and $1+z$ such that
\begin{align*}
    H(1+z) = H(1)+H'(1) z + \frac{H''(\xi)}{2}z^2.
\end{align*}
It follows from
the fact that $H''$ is monotone with $H''(1)= \frac{\gamma}{\eps^2}$ that
\begin{align}
\intx H(1+\eps^2 \bpv) \d x
\geq b(t)
\intx \eps^4 \bpv^2 \d x \label{eq:taylor_H}
\end{align}
with
\begin{align}
b(t) := \frac{1}{2} \min \big\{\frac{\gamma}{\eps^2}, \,\inf_{x\in [0, 1]} H''(1+\eps^2 \bpv)\big\}.
\end{align}
Note that, by \eqref{eq:density_bounds_infty_expectation}, there exists $C_{p, \gamma}>0$ such that,
for all sufficiently small $\eps>0$,
\begin{align*}
\ex \bigg[\Big(\supt b(t)^{-2p}\Big)^{\frac{1}{2}}\bigg] \leq C_{p, \gamma} \eps^{2p}.
\end{align*}
We use \eqref{eq:taylor_H} to compute that
\begin{align*}
\ex \bigg[\Big(\supt \intx \eps^4 \bpv^2 \d x  \Big)^p \bigg]
&\leq \ex \bigg[\Big(\supt  b(t)\intx \eps^4 \bpv^2 \d x  \Big)^{p}\bigg]
  \ex \bigg[\Big(\supt b(t)^{-2p}\Big)^{\frac{1}{2}}\bigg]\\
&\leq C\ex \bigg[\Big(\supt  \cE(\pv, \uv) \Big)^{p}\bigg]
  \ex \bigg[\Big(\supt b(t)^{-2p}\Big)^{\frac{1}{2}}\bigg]\\
&\leq C_{p, \gamma}\eps^2 \left(3^{2p-1}  \e^T \,
\ex [\cE(\p_{0, \eps},u_{0, \eps})^{2p}] + C_{2p} \, \e^T \, T^{2p+1} \eps^{4\beta p} \|\sigma\|^{4p}_{L^\infty}\right)^{\frac{1}{2}}\\
&\leq C_{p, \gamma}\eps^{2p}\big(\eps^{2p \alpha} + \eps^{4p\beta} \big)^{\frac{1}{2}} \\
&\leq C_{p, \gamma}\eps^{2p} \big(\eps^{p\alpha} + \eps^{2p\beta} \big),
\end{align*}
from which the bound on $\ex [\|\bpv\|_{L^\infty(0, T; L^2)}^{2p}]$ follows.
Regarding $\buv$, we similarly have
\begin{align*}
&\ex \bigg[\Big(\supt \eps^2\intx \buv^2 \d x \Big)^p\bigg]  \\
&\leq \ex \bigg[\Big(\supt \inf_{x \in [0,1]} (1+\eps^2 \bpv)\intx \eps^2 \buv^2 \d x  \Big)^{p}\bigg]
 \ex \bigg[\Big(\supt\big(\displaystyle\inf_{x \in [0,1]} (1+\eps^2 \bpv)\big)^{-1}\Big)^{p}\bigg]\\
&\leq C_p\bigg(\ex \Big[\Big(\supt  \cE(\pv, \uv)\Big)^{\!2p}\Big]\bigg)^{\frac{1}{2}}
\bigg(\ex \Big[\Big(\supt \big(\displaystyle \inf_{x \in [0,1]} (1+\eps^2 \bpv))^{-1}\Big)^{\!2p}\Big]\bigg)^{\frac{1}{2}}\\
&\leq C_p\big(\eps^{p\alpha} + \eps^{2p\beta} \big),
\end{align*}
giving the required bound on $\ex \|\buv\|_{L^\infty(0, T; L^2)}^{2p}$.
To complete the proof of \eqref{eq:enerone}, we need two more estimates on $(\bpv)_x$. The first is
\begin{align*}
\ex \bigg[ \Big( \intT \intx \eps^4|(\bpv)_x|^2 \d x \d t\Big)^p \bigg]
&\leq \ex \bigg[ \Big(\supt \|\pv^{3-\gamma}\|_{L^\infty}\Big)^p \Big(\intT \intx \pv^{\gamma-3}|(\pv)_x|^2\d x\d t\Big)^p\bigg]\\
&\leq  \bigg(\ex\Big[ \Big(\supt \|\pv^{3-\gamma}\|_{L^\infty}\Big)^{2p}\Big]\bigg)^{\frac{1}{2}}
 \bigg(\ex \Big[\Big(\intT \intx \pv^{\gamma-3} (\pv)_x^2 \d x \d t \Big)^{2p} \Big]\bigg)^{\frac{1}{2}}\\
&\leq C_p \big(\eps^{p\alpha} + \eps^{2p\beta} \big),
\end{align*}
while the second is
\begin{align*}
\ex \bigg[\Big(\supt \intx \eps^4 |(\bpv)_x|^2 \d x  \Big)^p\bigg]
&= \ex \bigg[\Big(\supt \intx |(\pv)_x|^2 \d x  \Big)^p \bigg]\\
&\leq \ex \bigg[\Big( \supt \|\pv\|_{L^\infty}^3 \intx \frac{|(\pv)_x|^2}{\pv^3} \d x  \Big)^p \bigg]\\
&\leq 8 \ex \bigg[\Big( \supt \|\pv\|_{L^\infty}^3 \cE (\pv, \uv)\Big)^p\bigg]\\
&\leq 8\bigg( \ex \Big[  \supt \|\pv\|_{L^\infty}^{6p}\Big]\bigg)^{\frac{1}{2}}
\bigg( \ex \Big[ \supt \cE (\pv, \uv)^{2p}\Big]\bigg)^{\frac{1}{2}}\\
&\leq C\big(\eps^{p\alpha} + \eps^{2p\beta} \big).
\end{align*}
We observe that, using estimate \eqref{eq:orig_poly_bound} again, we have
\begin{align*}
\ex \bigg[ \Big(\intT \intx \eps^2 |(\buv)_x|^2 \d x \d t\Big)^p \bigg]
= \ex \bigg[ \Big(\intT \intx |(\uv)_x|^2 \d x \d t \Big)^p\bigg]
\leq C_p \big(\eps^{p\alpha} + \eps^{2p\beta} \big),
\end{align*}
which gives the desired bound on $(\buv)_x$.

\smallskip
3. Finally, we turn to \eqref{2.8a},
and fix any $\delta>0$ complying with \eqref{eq:expectation_exp_u_delta}.
Then, by \eqref{eq:orig_expo_bound},
we have
\begin{align*}
\ex \Big[\exp ( \delta\intt \intx \buv^2 \d x  \d t) \Big]
&=\ex \Big[\exp(\frac{\delta}{\eps^2} \intt \intx u^2 \d x  \d t) \Big] \\
&\leq \frac{\e^T}{3} \ex \Big[\exp( \frac{3\delta}{\eps^2} \cE(\p_{0, \eps},u_{0, \eps}))\Big]
+\frac{C T\e^T}{1-\delta / (\eps^2 \eta_0(\eps, T))}\\
&\leq \frac{C \e^T}{3(1-3\delta\lambda \eps^{\alpha-2})}+\frac{C T\e^T}{1-\delta / (\eps^2 \eta_0(\eps, T))} \leq C.
\end{align*}
This completes the proof.
\end{proof}

\section{Analysis of the Acoustic System} \label{section:acoustic}

The starting point of our analysis of the acoustic system  \eqref{linear} is the existence and uniqueness of solutions
with the initial data  $(r_{0, \eps}, v_{0, \eps})$,
which are proved directly by using again a pathwise transformation:
$$
\hat{r}_\eps = \rv - \sum_{k = 1}^\infty \int_0^t\eps^{\beta-2} \sigma_k\, \d W^k_s.
$$

\begin{proposition} \label{prop:existence_linear}
Assume that $(r_{0, \eps}, v_{0, \eps})_{\eps \in (0, 1]}$ is a collection of $\mathcal{F}_0$-measurable
random variables in $H^1\times H^1_0$ such that
$$
\intx r_{0, \eps} \d x=0  \qquad \mbox{almost surely}.
$$
Then there exists a unique adapted strong pathwise solution $(\rv, \vv)$  of system \eqref{linear} such that
\begin{equation*}
\rv \in C([0, \infty); H^1), \quad \vv \in C([0, \infty); H^1_0) \cap L^2_{\mathrm{loc}}(0, \infty; H^2)
\qquad\mbox{almost surely}.
\end{equation*}
\end{proposition}

The energy structure of \eqref{linear} is captured by the energy functional:
\begin{equation*}
\mathcal{G}(r, v) = \intx \Big(\frac{1}{2}(\gamma r^2 + v^2) + \frac{1}{2}\eps v r_x
+ \frac{\eps^2}{4}|r_x|^2\Big)\d x,
\end{equation*}
which can be viewed as analogous to $\cE$ in \eqref{eq:def_modified energy}.

\subsection{Energy estimates}
We start by a series of the bounds on the solutions of \eqref{linear}.

\begin{proposition} \label{prop_poly_est_linear}
Let $(\rv, \vv)$ be the solution of system \eqref{linear} with initial data $(r_{0, \eps}, v_{0, \eps})$
as in {\rm Proposition \ref{prop:existence_linear}}, and let
\begin{equation*}
\kappa_\eps := \frac{1}{4\eps^{2\beta-2} \|\sigma\|_{L^\infty}^2}.
\end{equation*}
For all $R>0$, the following exponential martingale estimate holds{\rm :}
\begin{align}
&\mathbb{P}\Big[\sup_{t \geq 0} \Big(\mathcal{G}(\rv, \vv)
+ \frac{1}{4} \int_0^t \intx \big(|(\vv)_x|^2
+\gamma |(\rv)_x|^2\big) \d x\d s -\frac{1}{2} \eps^{2\beta-2} \|\sigma\|_{L^\infty}^2 t \Big)
- \mathcal{G}(r_{0, \eps}, v_{0, \eps}) \geq R \Big] \nonumber\\[1mm]
&\,\,\leq \e^{-\kappa_\eps R}. \label{eq:expmart}
\end{align}
As a consequence, for any $m \geq 1$, there exists $c_m>0$ independent of $\eps \in (0, 1]$
such that, for all $T \geq 1$,
\begin{align}
& \ex \Big[\sup_{0\leq t \leq T} \Big( \mathcal{G}(\rv, \vv)
+ \frac{1}{4} \int_0^t \intx \big(|(\vv)_x|^2+
|(\rv)_x|^2\big) \d x\d s \Big)^m \Big]\nonumber\\[1mm]
&\leq  c_m \left( \kappa_\eps^{-m}+  \eps^{2m\beta-2m} \|\sigma\|_{L^\infty}^{2m} T^m
+  \ex \left[\mathcal{G}(r_{0, \eps}, v_{0, \eps})^m \right]\right).\label{eq:polymart}
\end{align}
Moreover, for any $\delta\in (0, \frac{\kappa_\eps}{2})$,
there exist a constant $C_\delta>0$ independent of $\eps \in (0, 1]$ such that, for all $T \geq 1$,
\begin{align}
& \ex \Big[\exp \Big( \delta \supt\Big( \mathcal{G}(\rv, \vv)
+ \frac{1}{4} \int_0^t \intx \big(|(\vv)_x|^2
+\gamma |(\rv)_x|^2\big) \d x\d s\Big)\Big)\Big] \nonumber\\
&\leq  C_\delta \exp \big(\frac{\delta}{2}  \eps^{2\beta-2} \|\sigma\|_{L^\infty}^2 T\big)
\sqrt{ \ex\big[\exp \big(2\delta \mathcal{G}(r_{0, \eps}, v_{0, \eps})\big)\big]}.\label{eq:expmom}
\end{align}
\end{proposition}

\begin{proof}
By It\^{o}'s formula, the two equations in \eqref{linear} imply that
\begin{equation*} \label{eq:differetial_equality_of_energy}
\d \Big(\intx \frac{1}{2}(\gamma \rv^2 + \vv^2)\d x\Big)
+\Big(\intx |(\vv)_x|^2 \d x \Big) \d t
= \Big(\intx \frac{1}{2}\eps^{2\beta-2}\sigma^2 \d x\Big) \d t
+\Big( \intx \eps^{\beta-1}\sigma \vv \d x \Big) \d W_t.
\end{equation*}
Again, the first integral on the right-hand-side is
\begin{equation*}
    \intx \frac{1}{2}\eps^{2\beta-2}\sigma^2 \d x
    = \sum_{k=1}^\infty \intx \frac{1}{2}\eps^{2\beta-2}\sigma_k^2 \d x,
\end{equation*}
and the last integral is a short form representing
\begin{equation*}
\sum_{k=1}^\infty \Big( \intx \eps^{\beta-1}\sigma_k \vv \d x \Big) \d W^k_t.
\end{equation*}
This equality, together with the observation that
\begin{equation*}
\d \Big( \eps \vv (\rv)_x + \frac{\eps^2}{2}|(\rv)_x|^2\Big)
= -vv_{xx}\d t-\gamma |(\rv)_x|^2\d t + \eps^\beta (\rv)_x \sigma \d W_t,
\end{equation*}
gives
\begin{align}
&\d \mathcal{G}(\rv, \vv)
+\frac{1}{2}\Big(\intx \big(|(\vv)_x|^2+\gamma |(\rv)_x|^2\big) \d x \Big) \d t\nonumber \\
&= \Big(\intx \frac{1}{2}\eps^{2\beta-2}\sigma^2 \d x\Big) \d t
+\Big( \intx \eps^{\beta-1}\sigma \big(\vv+\frac{1}{2}\eps (\rv)_x\big) \d x \Big) \d W_t.
\label{eq:differetial_equality_of_modified_energy}
\end{align}
Define
\begin{align*}
Z(t) &:= \intt \Big( \intx \eps^{\beta-1}\sigma \big(\vv+\frac{1}{2}\eps (\rv)_x\big) \d x \Big) \d W_s \\
&= \sum_{k=1}^\infty \intt \Big( \intx \eps^{\beta-1}\sigma_k \big(\vv+\frac{1}{2}\eps (\rv)_x\big) \d x \Big) \d W^k_s.
\end{align*}
Then $(Z(t))_{t \geq 0}$ is a (local) martingale. It follows from \eqref{eq:differetial_equality_of_modified_energy}
that, for all $t \geq 0$,
\begin{align}
&\mathcal{G}(\rv, \vv)
+\frac{1}{2} \int_0^t \intx \big(|(\vv)_x|^2+\gamma |(\rv)_x|^2\big) \d x \d s \nonumber\\[1mm]
&= \mathcal{G}(r_{0, \eps}, v_{0, \eps})
+ \frac{1}{2}\eps^{2\beta-2}\int_0^t \intx \sigma^2 \d x \d s + Z(t).\label{eq:energy_balance_linear}
\end{align}
Also, the quadratic variation of $Z(t)$ satisfies
\begin{align*}
\langle Z\rangle(t)
&= \sum_{k=1}^\infty \int_0^t \Big( \intx \eps^{\beta-1}\sigma_k \big(\vv+\frac{1}{2}\eps (\rv)_x\big)\d x \Big)^2 \d s\\
&\leq \eps^{2(\beta-1)}  \sum_{k=1}^\infty \int_0^t
  \|\sigma_k\|_{L^2}^2
  \big\|\vv+\frac{1}{2}\eps (\rv)_x \big\|_{L^2}^2 \d s\\
&\leq 2\eps^{2(\beta-1)}\|\sigma\|_{L^\infty}^2 \int_0^t \|\vv\|_{L^2}^2 \d s
 + \frac{1}{2}\eps^{2\beta}\|\sigma\|_{L^\infty}^2 \int_0^t \intx |(\rv)_x|^2 \d x \d s\\
&\leq \frac{1}{2\kappa_\eps}
 \intt \intx \big(|(\vv)_x|^2
 + \gamma |(\rv)_x|^2\big) \d x \d s.
\end{align*}
Let
\begin{equation*}
\Phi(t) := \mathcal{G}(\rv, \vv)(t) + \frac{1}{4} \int_0^t \intx
\big(|(\vv)_x|^2
+\gamma |(\rv)_x|^2\big) \d x\d s.
\end{equation*}
Then
\begin{align*}
&  \Phi(t) -\frac{1}{2} \eps^{2\beta-2} \|\sigma\|_{L^\infty}^2 t - \mathcal{G}(r_{0, \eps}, v_{0, \eps}) \\
&\leq  Z(t) - \frac{\kappa_\eps}{2}\langle Z\rangle(t)
 + \Big(\frac{\kappa_\eps}{2}\langle Z\rangle(t)
- \frac{1}{4} \int_0^t \intx \big(|(\vv)_x|^2
+\gamma |(\rv)_x|^2\big) \d x\d s\Big)\\
&\leq Z(t) - \frac{\kappa_\eps}{2}\langle Z\rangle(t).
\end{align*}
Thus, it follows from the exponential martingale estimates that, for all $R > 0$,
\begin{align*}
 &\mathbb{P} \Big[\sup_{t \geq 0} \big(\Phi(t)-\frac{1}{2} \eps^{2\beta-2} \|\sigma\|_{L^\infty}^2 t\big)
  - \mathcal{G}(r_{0, \eps}, v_{0, \eps}) \geq R \Big]\\
 &\leq \mathbb{P} \Big[\sup_{t \geq 0} \big(Z(t) - \frac{\kappa_\eps}{2}\langle Z\rangle(t)(t) \big) \geq R \Big] \\
&\leq \e^{-\kappa_\eps R},
\end{align*}
which is \eqref{eq:expmart}.
To prove \eqref{eq:polymart}, fix $T>0$ and denote
$$
Y := \frac{1}{2} \eps^{2\beta-2} \|\sigma\|_{L^\infty}^2 T +  \mathcal{G}(r_{0, \eps}, v_{0, \eps}).
$$
We also denote by $\chi$ the characteristic function of the set: $\{\supt \Phi(t) >Y\}$.
Then we have
\begin{align*}
 \ex \Big[ \supt \Phi(t)^m \Big]
&\leq 2^{m-1}\ex  \Big[ \supt(\Phi(t)-Y)^m \chi \Big]
 +2^{m-1}\ex  \big[ Y^m \chi \big]
 + \ex  \Big[ \supt \Phi(t)^m  (1-\chi) \Big]\\
&\leq 2^{m-1}\ex  \Big[ \supt (\Phi(t)-Y)^m \chi \Big] +2^{m-1}\ex [Y^m]\\
&= 2^{m-1}\int_0^\infty  \mathbb{P}\Big[ \supt(\Phi(t)-Y)^m \chi \geq \lambda \Big] \d\lambda
  +2^{m-1} \ex \big[Y^m\big]\\
&\leq  2^{m-1}\int_0^\infty \e^{-\kappa_\eps \lambda^{1/m}} \d\lambda
   +2^{m-1} \ex \Big[\Big(\frac{1}{2} \eps^{2\beta-2} \|\sigma\|_{L^\infty}^2 T +  \mathcal{G}(r_{0, \eps}, v_{0, \eps})\Big)^m\Big]\\
&\leq  c_m \Big( \kappa_\eps^{-m}+  \eps^{2m\beta-2m} \|\sigma\|_{L^\infty}^{2m} T^m
+  \ex \big[\mathcal{G}(r_{0, \eps}, v_{0, \eps})^m \big]\Big),
\end{align*}
as we want.
Turning to \eqref{eq:expmom},
we fix $0 < \delta < \frac{\kappa_\eps}{2}$ and argue similarly to obtain
\begin{align*}
\ex \Big[\exp \big( \delta\supt \Phi(t)\big)\Big]
&\leq \ex \Big[\exp \big( \delta\supt (\Phi(t)- Y) \chi \big)\exp (\delta Y \chi)
 \exp \big( \delta(1-\chi)\supt \Phi(t) \big)\Big]\\
&\leq \ex \Big[ \exp ( \delta Y) \exp \big( \delta\supt (\Phi(t)-Y) \chi \big) \Big]\\
&\leq \sqrt{\ex \big[ \exp( 2\delta Y)\big]}
  \sqrt{ \ex\Big[ \exp \big( 2\delta\supt (\Phi(t)-Y) \chi \big) \Big]}\\
&\leq \sqrt{\ex \big[ \exp ( 2\delta Y)\big]}
\sqrt{\int_1^\infty \mathbb{P}\Big[\supt (\Phi(t)-Y) \chi \geq \frac{1}{2\delta} \log \lambda\Big] \d\lambda}\\
&\leq \sqrt{\ex \Big[\exp \big(2\delta( \frac{1}{2} \eps^{2\beta-2} \|\sigma\|_{L^\infty}^2 T
+  \mathcal{G}(r_{0, \eps}, v_{0, \eps}) ) \big)\Big]} \sqrt{\int_1^\infty \lambda^{-\kappa_\eps/(2\delta)} \d\lambda}\\
&= C_\delta \exp \big(\frac{\delta}{2}  \eps^{2\beta-2} \|\sigma\|_{L^\infty}^2 T \big)
\sqrt{ \ex\big[ \exp \big(2\delta \mathcal{G}(r_{0, \eps}, v_{0, \eps})\big)\big]}.
\end{align*}
This concludes the proof.
\end{proof}

\subsection{Exponential contraction and stability}
Due to the linearity of system \eqref{linear} and the additivity of the noise term,
the stability of the acoustic system can be analyzed through the corresponding deterministic system.
The main result of this section is the following{\rm :}

\begin{proposition}
Let $\eps \in (0,1]$.
Let $(\rv, \vv) \in C([0, \infty); H^1) \times C([0, \infty; H^1_0) \cap L^2_{\text{\rm loc}}(0, \infty; H^2)$
be a solution to the deterministic system:
\begin{equation} \label{eq:main_linear_deterministic}
\begin{cases}
(\rv)_t + \frac{1}{\eps} (\vv)_x =0,\\
(\vv)_t + \frac{\gamma}{\eps} (\rv)_x = (\vv)_{xx} .
\end{cases}
\end{equation}
Then the following exponential stability estimates hold{\rm :} For every $t\geq 0$,
\begin{align}
&\mathcal{G}(\rv(t, \cdot), \vv(t, \cdot)) \leq \mathcal{G}(\rv(0, \cdot), \vv(0, \cdot))\e^{-\frac{t}{2}}, \label{eq:exp_decay_modified_energy}\\[1mm]
&\|(\rv,\, \vv)(t, \cdot)\|_{H^1}^2
\leq 8 \gamma \|(\rv, \,\vv)(0, \cdot)\|_{H^1}^2
\e^{-\frac{t}{4}}, \label{eq:exp_decay_H_1}\\[2mm]
&\|(\rv,\,\vv)(t,\cdot)\|_{L^2}^2
+ \eps^2 \|((\rv)_x,
(\vv)_x)(t,\cdot)\|_{L^2}^2 \nonumber \\
& \quad \leq 6\gamma \Big( \|(\rv,\,\vv)(0,\cdot)\|_{L^2}^2
+ \eps^2 \|((\rv)_x, (\vv)_x)(0,\cdot)\|_{L^2}^2 \Big) \e^{-\frac{t}{2}}.   \label{eq:exp_decay_epsilon_H1}
\end{align}
\end{proposition}

\begin{proof} First, notice the following pointwise inequality:
\begin{align*}
\frac{1}{2}(\gamma \rv^2 + \vv^2) + \frac{1}{2}\eps \vv (\rv)_x + \frac{\eps^2}{4}|(\rv)_x|^2
\leq \frac{\gamma}{2}\rv^2 +\frac{3}{4} \vv^2+  \frac{\eps^2}{2}|(\rv)_x|^2.
\end{align*}
Then, applying the Poincar\'e inequality for both $\rv$ and $\vv$ in the second inequality, we have
\begin{align}
\mathcal{G}(\rv, \vv) &\leq \frac{\gamma +\eps^2}{2}\|(\rv)_x\|_{L^2}^2 +\frac{3}{4} \|\vv\|_{L^2}^2
\leq \gamma \|(\rv)_x\|_{L^2}^2 +\|(\vv)_x\|_{L^2}^2  \label{eq:bound_G_by_H1}
\end{align}
for all $\eps \in (0, 1]$.
By the equations, we have the following equality (see \eqref{eq:differetial_equality_of_modified_energy}):
\begin{align} \label{eq:energy_bal_linear_deterministic}
\ddt \mathcal{G}(\rv, \vv) = -\frac{\gamma}{2}\|(\rv)_x\|_{L^2}^2 -\frac{1}{2}\|(\vv)_x\|_{L^2}^2,
\end{align}
which implies
\begin{align*}
\ddt \mathcal{G}(\rv, \vv) \leq -\frac{1}{2}\mathcal{G}(\rv, \vv).
\end{align*}
The first estimate \eqref{eq:exp_decay_modified_energy} then follows from the Gronwall inequality.

Now, differentiating the first equation in \eqref{eq:main_linear_deterministic} with respect to $x$ gives
\begin{align*}
((\rv)_x)_t + \frac{1}{\eps}(\vv)_{xx} =0.
\end{align*}
With this, multiplying the second equation of \eqref{eq:main_linear_deterministic} by $(\vv)_{xx}$ and integrating in $x$,
we have
\begin{align*}
\ddt \intx \frac{1}{2}|(\vv)_x|^2 \d x + \intx |(\vv)_{xx}|^2 \d x
= \intx \frac{\gamma}{\eps}(\rv)_x (\vv)_{xx}\d x= -\ddt \intx \frac{\gamma}{2}|(\rv)_x|^2 \d x,
\end{align*}
which implies
\begin{align} \label{eq:linear_derivation_H1}
\ddt \Big( \intx \big(\gamma |(\rv)_x|^2+|(\vv)_x|^2\big) \d x \Big)
&= -2\intx |(\vv)_{xx}|^2 \d x \leq -2 \intx |(\vv)_x|^2 \d x.
\end{align}
The last inequality is due to the Poincar\'e inequality applied to $(\vv)_x$ that is of zero mean.
Combining this with \eqref{eq:energy_bal_linear_deterministic}, we have
\begin{align*}
\ddt \Big( \mathcal{G}(\rv, \vv)
+ \intx \big(\gamma |(\rv)_x|^2+|(\vv)_x|^2\big) \d x \Big)
&\leq -\frac{\gamma}{2}\|(\rv)_x\|_{L^2}^2 -\frac{5}{2}\|(\vv)_x\|_{L^2}^2\\
&\leq -\frac{1}{4}\Big( \mathcal{G}(\rv, \vv) + \intx \big(\gamma|(\rv)_x|^2+|(\vv)_x|^2\big) \d x \Big).
\end{align*}
The second estimate \eqref{eq:exp_decay_H_1} now follows from the Gronwall inequality and the bounds:
\begin{align*}
 \frac{1}{4}\intx  \big(\gamma \rv^2+\vv^2+\gamma |(\rv)_x|^2+|(\vv)_x|^2\big) \d x
 &\leq  \mathcal{G}(\rv, \vv) + \intx \big(\gamma |(\rv)_x|^2+|(\vv)_x|^2\big) \d x \\
 &\leq 2 \intx \big(\gamma |(\rv)_x|^2+|(\vv)_x|^2\big)\d x.
\end{align*}

Similarly, \eqref{eq:energy_bal_linear_deterministic}--\eqref{eq:linear_derivation_H1} imply
\begin{align*}
\ddt \Big( \mathcal{G}(\rv, \vv) + \eps^2 \intx \big(\gamma |(\rv)_x|^2+|(\vv)_x|^2\big) \d x \Big)
&\leq -\frac{\gamma}{2}\|(\rv)_x\|_{L^2}^2 -\frac{1}{2}\|(\vv)_x\|_{L^2}^2\\
&\leq -\frac{1}{2}\Big( \mathcal{G}(\rv, \vv)
+ \eps^2 \intx \big(\gamma|(\rv)_x|^2+|(\vv)_x|^2\big) \d x \Big).
\end{align*}
The last estimate \eqref{eq:exp_decay_epsilon_H1} now follows from the Gronwall inequality and the bounds:
\begin{align*}
 \frac{1}{4}\intx  \big(\rv^2+\vv^2+\eps|(\rv)_x|^2+ \eps|(\vv)_x|^2\big) \d x
 &\leq  \mathcal{G}(\rv, \vv) + \eps^2 \intx \big(\gamma |(\rv)_x|^2+|(\vv)_x|^2\big) \d x \\
 &\leq \frac{3\gamma}{2}\intx \big(\rv^2 + \vv^2 + \eps|(\rv)_x|^2+\eps|(\vv)_x|^2\big)\d x.
\end{align*}
This concludes the proof.
\end{proof}

A straightforward consequence of the above result is the stability and continuous dependence of solutions of problem \eqref{ID-2}--\eqref{BC-2}
for system \eqref{linear},
which follows from the observation that the difference of
any two solutions of system \eqref{linear} is a solution to
the deterministic system \eqref{eq:main_linear_deterministic}.

\begin{corollary} \label{cor_linear_exponential_approach}
Let $\eps \in (0,1]$. Let $(r_{1, \eps}, v_{1, \eps})$ and $(r_{2, \eps}, v_{2, \eps})$
be two solutions of  problem \eqref{ID-2}--\eqref{BC-2} for the stochastic system \eqref{linear}.
Then, for all $t \geq 0$,
\begin{align*}
& \big\|(r_{1, \eps}-r_{2, \eps},\,v_{1, \eps}-v_{2, \eps})(t,\cdot)\big\|_{H^1}^2
\leq 8 \gamma \big\|(r_{1, \eps}-r_{2, \eps},\, v_{1, \eps}-v_{2, \eps})(0,\cdot)\big\|_{H^1}^2 \e^{-\frac{t}{4}},\\[2mm]
& \big\|(r_{1, \eps}-r_{2, \eps},\,v_{1, \eps}-v_{2, \eps})(t,\cdot)\big\|_{L^2}^2
+ \eps^2 \big\|((r_{1, \eps})_x-(r_{2, \eps})_x, \,(v_{1, \eps})_x -(v_{2, \eps})_x) (t,\cdot)\big\|_{L^2}^2  \\
&\,\,\,\leq 6 \gamma \Big(\big\|(r_{1, \eps}-r_{2, \eps},\,v_{1, \eps}-v_{2, \eps})(0,\cdot)\big\|_{L^2}^2
+ \eps^2 \big\|((r_{1, \eps})_x -(r_{2, \eps})_x,\,(v_{1, \eps})_x -(v_{2, \eps})_x)(0,\cdot)\big\|_{L^2}^2\Big) \e^{-\frac{t}{2}},
\end{align*}
almost surely.
\end{corollary}

\subsection{Long-time behavior}
We can further explore the dissipative property of system \eqref{linear}
to establish the following energy dissipative estimates.

\begin{proposition} \label{prop:large_time_linear}
Let $(\rv, \vv)$ be the solution of system \eqref{linear} with initial data $(r_{0, \eps}, v_{0, \eps})$.
Then, for all $p\geq 1$, there exists $C_p>0$ independent of $\eps$ such that, for all $t\geq 0$,
\begin{align}\label{eq:largepmom}
\ex \big[\mathcal{G}(\rv(t, \cdot), \vv(t, \cdot))^p \big]
\leq C_p \Big(  \ex \big[\e^{-\frac{p t}{2}}\mathcal{G}(r_{0, \eps}, v_{0, \eps})^p \big]
+ \eps^{2(\beta-1)p} \|\sigma\|_{L^\infty}^{2p}\Big).
\end{align}
Moreover,
for all $0 < \eta < (8\eps^{2(\beta-1)} \|\sigma\|_{L^\infty}^2)^{-1}$, there exists $C_\eta>0$ independent of $\eps$ such that, for all $t \geq 0$,
\begin{align}\label{eq:largeexpmom}
\ex \big[ \exp\big(\eta \mathcal{G}(\rv(t, \cdot), \vv(t, \cdot)\big)\big]
\leq C_\eta \exp\big(2\eta\eps^{2(\beta-1)}\big) \sqrt{\ex \big[ \exp\big(2 \e^{-\frac{t}{4}}\mathcal{G}(r_{0, \eps}, v_{0, \eps})\big) \big]}.
\end{align}
\end{proposition}

\begin{proof}
From \eqref{eq:differetial_equality_of_modified_energy} and \eqref{eq:bound_G_by_H1}, we obtain
\begin{align*}
&\d \big(\e^{\frac{t}{4}} \mathcal{G}(\rv, \vv)\big)
+\frac{1}{4} \Big( \e^{\frac{t}{4}} \intx \big(|(\vv)_x|^2+ |(\rv)_x|^2\big) \d x \Big) \d t \\
&\leq \Big(\e^{\frac{t}{4}}\intx \frac{1}{2}\eps^{2(\beta-1)}\sigma^2 \d x\Big) \d t
+\Big( \e^{\frac{t}{4}} \intx \eps^{\beta-1}\sigma \big(\vv+\frac{1}{2}\eps (\rv)_x\big) \d x \Big) \d W_t.
\end{align*}
Hence, for all $t \geq 0$, we have
\begin{align*}
\mathcal{G}(\rv(t, \cdot), \vv(t, \cdot))
&\leq \e^{-\frac{t}{4}}\mathcal{G}(r_{0, \eps}, v_{0, \eps}) + 2\eps^{2(\beta-1)}(1-\e^{-\frac{t}{4}})\|\sigma\|_{L^\infty}^2 \\
& \quad +\eps^{\beta-1} \intt \e^{-\frac{t-s}{4}}
  \intx \sigma \big(\vv+\frac{1}{2}\eps (\rv)_x\big) \d x \d W_s\\
& \quad -\frac{1}{4} \Big( \intt \e^{-\frac{t-s}{4}}
\intx \big(|(\vv)_x|^2+ |(\rv)_x|^2\big) \d x \Big) \d s.
\end{align*}
Therefore, for all $t \geq 0$ and $R>0$, we deduce
\begin{align*}
&\mathbb{P}\Big[\mathcal{G}(\rv(t, \cdot), \vv(t, \cdot))-\e^{-\frac{t}{2}}\mathcal{G}(r_{0, \eps}, v_{0, \eps})
- 2\eps^{2(\beta-1)}(1-\e^{-\frac{t}{4}})\|\sigma\|_{L^\infty}^2 \geq R  \Big]\\
&\leq \mathbb{P}\Big[\eps^{\beta-1} \intt \e^{-\frac{t-s}{2}} \intx \sigma \big(\vv+\frac{1}{2}\eps (\rv)_x\big) \d x \d W_s
-\frac{1}{4} \intt \e^{-\frac{t-s}{4}} \intx \big(|(\vv)_x|^2+ |(\rv)_x|^2\big) \d x  \d s \geq R \Big] \\
&\leq \e^{-\kappa_\eps R},
\end{align*}
where $\kappa_\eps = (4\eps^{2(\beta-1)} \|\sigma\|_{L^\infty}^2)^{-1}$ is the same in Proposition \ref{prop_poly_est_linear}.
The last inequality is due to \cite{Mattingly02}*{Lemma A.1} and the fact that
\begin{align*}
\sum_{k = 1}^\infty \Big( \intx \eps^{\beta -1} \sigma_k \big(\vv+\frac{1}{2}\eps (\rv)_x\big) \d x \Big)^2
\leq 2 \eps^{2(\beta-1)}\|\sigma\|_{L^\infty}^2 \Big( \|(\vv)_x\|_{L^2}^2 + \frac{1}{4} \eps^{2} \|(\rv)_x\|_{L^2}^2\Big).
\end{align*}
As a result, we use a computation similar to the proof of Proposition \ref{prop_poly_est_linear} to obtain that,
for all $p \geq 1$,
\begin{align*}
&\ex \Big[\big(\mathcal{G}(\rv(t, \cdot), \vv(t, \cdot))-\e^{-\frac{t}{2}}\mathcal{G}(r_{0, \eps}, v_{0, \eps})
- 2\eps^{2(\beta-1)}(1-\e^{-\frac{t}{4}})\|\sigma\|_{L^\infty}^2\big)^p \Big]\\
&\leq p \int_0^\infty \lambda^{p-1} \e^{-\kappa_\eps \lambda} \d \lambda
= \kappa_\eps^{-p} p! \, p.
\end{align*}
Consequently, we have
\begin{align*}
&\ex \big[\mathcal{G}(\rv(t, \cdot), \vv(t, \cdot))^p\big]\\
&\qquad\leq 3^{p-1} \ex \Big[\e^{-\frac{p t}{2}}\mathcal{G}(r_{0, \eps}, v_{0, \eps})^p \Big]
+3^{p-1}\,  2^p \, \eps^{2(\beta-1)p}(1-\e^{-\frac{t}{4}})^p \|\sigma\|_{L^\infty}^{2p}\\
&\qquad \quad + 3^{p-1} \ex \Big[\big(\mathcal{G}(\rv(t, \cdot), \vv(t, \cdot))-\e^{-\frac{t}{4}}\mathcal{G}(r_{0, \eps}, v_{0, \eps})
- 2\eps^{2(\beta-1)}(1-\e^{-\frac{t}{4}})\|\sigma\|_{L^\infty}^2\big)^p \Big]\\
&\qquad= 3^{p-1} \ex \Big[\e^{-\frac{p t}{2}}\mathcal{G}(r_{0, \eps}, v_{0, \eps})^p \Big]
+3^{p-1}\,  2^p \, \eps^{2(\beta-1)p}(1-\e^{-\frac{t}{4}})^p \|\sigma\|_{L^\infty}^{2p} +\kappa_\eps^{-p} p! \, p,
\end{align*}
which leads to \eqref{eq:largepmom}.
Likewise, we can deduce that, for all $\eta\in (0, \kappa_\eps)$,
\begin{align*}
&\ex \Big[\exp \Big( \eta \big(\mathcal{G}(\rv(t, \cdot), \vv(t, \cdot))-\e^{-\frac{t}{4}}\mathcal{G}(r_{0, \eps}, v_{0, \eps})
- 2\eps^{2(\beta-1)}(1-\e^{-\frac{t}{4}})\|\sigma\|_{L^\infty}^2\big)\Big) \Big]\\
&\leq 1+\int_1^\infty \lambda^{-\frac{\kappa_\eps}{\eta}} \d \lambda = \frac{\kappa_\eps}{\kappa_\eps-\eta}.
\end{align*}
Hence, for all $t \geq 0$ and $\eta\in (0,\frac{\kappa_\eps}{2})$,
\begin{equation*}
 \ex \big[ \exp\big(\eta \mathcal{G}(\rv(t, \cdot), \vv(t, \cdot)\big)\big]
 \leq \exp\big(2\eta\eps^{2(\beta-1)}(1-\e^{-\frac{t}{4}})\big)
 \sqrt{\frac{\kappa_\eps}{\kappa_\eps-2\eta}\ex \big[ \exp\big(2 \e^{-\frac{t}{4}}\mathcal{G}(r_{0, \eps}, v_{0, \eps})\big) \big]},
\end{equation*}
from which \eqref{eq:largeexpmom} follows. This completes the proof.
\end{proof}

\section{Finite-Time Convergence: Proof of Theorem \ref{thm:finite_time_convergence}}\label{section:finite_time_analysis}

In this section, we gather the estimates established in the previous sections to provide a proof of
Theorem \ref{thm:finite_time_convergence}. We first estimate the difference between $(\bpv, \buv)$ and $(\rv, \vv)$.
\begin{proposition}
There exists a generic constant $C>0$, independent of any parameters, such that, for all $t \geq 0$, the following
estimate holds almost surely{\rm :}
\begin{align}
&\gamma \|\bpv(t)-\rv(t)\|_{L^2}^2 + \|\buv(t)-\vv(t)\|_{L^2}^2 + \intt \|(\buv-\vv)_x\|_{L^2}^2 \d s\nonumber\\
&\leq \exp \Big(\intt \frac{\eps}{\gamma}\|(\buv)_x \|_{L^2}^2 \d s\Big)\nonumber\\
&\quad\times\Big(\gamma\|\bpv(0)-\rv(0)\|_{L^2}^2 + \|\buv(0)-\vv(0)\|_{L^2}^2 +C \eps^2 \int_0^t \intx \buv^4 \d x \d s\nonumber\\
&\qquad \quad + C\eps^4 \int_0^t \|(1+\eps^2 \bpv)^{-2}\|_{L^\infty}^2
   \|(\bpv)_x \|_{L^2}^2 \|(\buv)_x\|_{L^2}^2 \d x \d s\nonumber\\
&\qquad \quad + C |\gamma-2| \eps^2 \int_0^t \|(1+\eps^2 \bpv)^{\gamma-3}\|_{L^\infty} \intx |\bpv|^4 \d x \d s
+  C\gamma^2 \eps \int_0^t \|(\bpv)_x\|_{L^2}^2 \d s\Big).\label{eq:difference_L^2}
\end{align}
\end{proposition}

\begin{proof}
Taking the difference of the first equation of \eqref{eq:nonlinear} and the first equation of \eqref{linear},
and multiplying it by $(\bpv -\rv)$, we have
\begin{align*}
\frac{1}{2}\big((\bpv-\rv)^2\big)_t + \frac{1}{\eps} (\buv-\vv)_x(\bpv-\rv) = -\eps (\bpv \buv)_x (\bpv-\rv).
\end{align*}
Likewise, taking the difference of the second equation of \eqref{eq:nonlinear} and the second equation of \eqref{linear},
and multiplying it with $(\buv-\vv)$, we obtain
\begin{align*}
&\frac{1}{2}\big((\buv-\vv)^2\big)_t + \frac{\gamma}{\eps} (\bpv-\rv)_x(\buv-\vv) \\
&= (\buv-\vv)_{xx}(\buv-\vv) + \frac{\gamma}{\eps}\big(1-(1+\eps^2 \bpv)^{\gamma-2}\big)(\bpv)_x (\buv-\vv) \\
&\quad -\eps^2 \frac{\bpv (\buv)_{xx}}{1+\eps^2	\bpv} (\buv-\vv) - \eps uu_x(\buv-\vv).
\end{align*}
Summing the two equations and integrating in $x\in (0, 1)$, we obtain
\begin{align*}
&\ddt \big(\gamma \|\bpv-\rv\|_{L^2}^2 + \|\buv-\vv\|_{L^2}^2\big) + 2\|(\buv-\vv)_x\|_{L^2}^2\\
&= -2\gamma \eps \intx (\bpv-\rv)(\bpv)_x \buv \d x -2\gamma\eps \intx (\bpv-\rv)\bpv (\buv)_x \d x\\
& \quad+ \frac{2\gamma}{\eps} \intx \big(1-(1+\eps^2 \bpv)^{\gamma-2}\big)(\bpv)_x (\buv-\vv) \d x \\
& \quad+\eps\intx \buv^2(\buv-\vv)_x \d x-2\eps^2 \intx \frac{\bpv}{1+\eps^2\bpv}(\buv)_{xx}(\buv-\vv)\d x\\
&=: \sum_{i=1}^5 I_i.
\end{align*}
Notice that
\begin{align*}
I_1+I_2
&=-2\gamma\eps \intx (\bpv-\rv)(\bpv)_x \buv \d x -2\gamma\eps \intx (\bpv-\rv)\bpv (\buv)_x \d x\\
&\leq 2\gamma\eps\|\bpv-\rv\|_{L^2} \|(\bpv)_x\|_{L^2} \|\buv\|_{L^\infty} + 2\gamma\eps\|\bpv-\rv\|_{L^2} \|\bpv\|_{L^\infty} \|(\buv)_x\|_{L^2}\\
&\leq 4 \gamma\eps\|\bpv-\rv\|_{L^2} \|(\bpv)_x\|_{L^2} \|(\buv)_x\|_{L^2}\\
&\leq \eps\|(\buv)_x\|_{L^2}^2\|\bpv-\rv\|_{L^2}^2 + C\gamma^2 \eps \|(\bpv)_x\|_{L^2}^2.
\end{align*}
Next, we estimate $I_3$. For $\gamma=1$,
\begin{align*}
I_3 &\leq2\eps \left|\intx \frac{\bpv(\bpv)_x}{1+\eps^2 \bpv}(\buv-\vv)\d x\right|\\
&=2\eps\left| \intx \Big(\frac{1}{\eps^2}\bpv-\frac{1}{\eps^4}\log(1+\eps^2\bpv)\Big)_x (\buv-\vv) \d x \right|\\
&=2\eps\left| \intx \Big(\frac{1}{\eps^2}\bpv-\frac{1}{\eps^4}\log(1+\eps^2\p)\Big) (\buv-\vv)_x\d x\right|\\
&\leq \frac{1}{4} \|(\buv-\vv)_x\|_{L^2}^2
  + C\eps^2 \intx \Big(\frac{1}{\eps^2}\bpv-\frac{1}{\eps^4}\log(1+\eps^2\p)\Big)^2 \d x \\
&\leq \frac{1}{4} \|(\buv-\vv)_x\|_{L^2}^2 + C\eps^2 \intx \frac{|\bpv|^4}{(1+\eps^2 \bpv)^2} \d x \\
&\leq \frac{1}{4} \|(\buv-\vv)_x\|_{L^2}^2 + C\eps^2 \|(1+\eps^2 \bpv)^{-1}\|_{L^\infty}^2 \intx |\bpv|^4 \d x,
\end{align*}
where we used here the inequality:
$\frac{1}{\eps^2}x-\frac{1}{\eps^4}\log(1+\eps^2x) \leq \frac{x^2}{1+\eps^2 x}$,
which is true for all $x$ with $1+\eps^2 x>0$.

\smallskip
\noindent
For $\gamma >1$ and $\gamma \neq 2$, we proceed as
\begin{align*}
I_3 &\leq\frac{2\gamma}{\eps}\left| \intx \big(1-(1+\eps^2 \bpv)^{\gamma-2}\big)(\bpv)_x (\buv-\vv) \d x\right|\\
&=\frac{2\gamma}{\eps^3}\left| \intx \Big(1+\eps^2\bpv-\frac{1}{\gamma-1}(1+\eps^2\bpv)^{\gamma-1}\Big)_x (\buv-\vv)
\d x\right|\\
&=\frac{2\gamma}{\eps^3} \left|\intx \Big(1+\eps^2\bpv-\frac{1}{\gamma-1}(1+\eps^2\bpv)^{\gamma-1}\Big)(\buv-\vv)_x \d x\right|\\
&\leq \frac{1}{4} \|(\buv-\vv)_x\|_{L^2}^2
+ \frac{C}{\eps^6} \intx \Big(1+\eps^2 \bpv-\frac{1}{\gamma-1}(1+\eps^2\bpv)^{\gamma-1}\Big)^2 \d x \\
&\leq \frac{1}{4} \|(\buv-\vv)_x\|_{L^2}^2 + C|\gamma-2| \eps^2 \|(1+\eps^2 \bpv)^{\gamma-3}\|_{L^\infty} \intx |\bpv|^4 \d x.
\end{align*}
Lastly, for $\gamma =2$, $I_3=0$.

\smallskip
\noindent
In all cases, we have
\begin{equation*}
I_3 \leq \frac{1}{4} \|(\buv-\vv)_x\|_{L^2}^2
+ C|\gamma-2| \eps^2\, \|(1+\eps^2 \bpv)^{\gamma-3}\|_{L^\infty} \intx |\bpv|^4 \d x.
\end{equation*}

The fourth integral can be bounded as
\begin{equation}
I_4=\eps\intx \buv^2(\buv-\vv)_x \d x\leq \frac{1}{8}  \|(\buv-\vv)_x\|_{L^2}^2 + C\eps^2 \intx |\buv|^4 \d x.
\end{equation}

Finally, we have
\begin{align*}
I_5 &= -2\eps^2 \intx \frac{\bpv}{1+\eps^2\bpv}(\buv)_{xx}(\buv-\vv)\d x\\
&=2\eps^2 \intx \frac{(\bpv)_x}{(1+\eps^2\bpv)^2}(\buv)_{x}(\buv-\vv)\d x+2\eps^2 \intx \frac{\bpv}{1+\eps^2\bpv}(\buv)_{x}(\buv-\vv)_x\d x\\
&\leq \frac{1}{16}\|\buv-\vv\|_{L^\infty}^2 + C \eps^4 \|(1+\eps^2 \bpv)^{-2}\|_{L^\infty}^2\|(\bpv)_x\|_{L^2}^2 \|(\buv)_x\|_{L^2}^2 \\
&\quad \, + \frac{1}{16} \|(\buv-\vv)_x\|_{L^2}^2 + C\eps^4 \|(1+\eps^2 \bpv)^{-1}\|_{L^\infty}^2 \|\bpv\|_{L^\infty}^2  \|(\buv)_x\|_{L^2}^2  \\
&\leq \frac{1}{8}\|(\buv-\vv)_x\|_{L^2}^2 + C \eps^4 \|(1+\eps^2 \bpv)^{-1}\|_{L^\infty}^2 \|(\buv)_x\|_{L^2}^2 \big(\|(\bpv)_x\|_{L^2}^2 + \|\bpv\|_{L^\infty}^2\big) \\
&\leq \frac{1}{8}\|(\buv-\vv)_x\|_{L^2}^2 + C \eps^4 \|(1+\eps^2 \bpv)^{-2}\|_{L^\infty}^2 \|(\buv)_x\|_{L^2}^2\|(\bpv)_x\|_{L^2}^2.
\end{align*}

Combining all the above, we conclude the following inequality almost surely:
\begin{align*}
& \ddt \Big(\gamma\|\bpv-\rv\|_{L^2}^2 + \|\buv-\vv\|_{L^2}^2\Big) +\|(\buv-\vv)_x\|_{L^2}^2\\[1mm]
&\leq \eps\|(\buv)_x \|_{L^2}^2\|\bpv-\rv\|_{L^2}^2  +  C\eps^4 \|(1+\eps^2 \bpv)^{-2}\|_{L^\infty}^2 \|(\bpv)_x \|_{L^2}^2 \|(\buv)_x\|_{L^2}^2 \\
&\quad + C|\gamma-2|\eps^2 \|(1+\eps^2 \bpv)^{\gamma-3}\|_{L^\infty} \intx |\bpv|^4 \d x +  C \gamma^2 \eps \|(\bpv)_x\|_{L^2}^2 +C  \eps^2 \intx |\buv|^4 \d x\\
&\leq \frac{\eps}{\gamma}\|(\buv)_x \|_{L^2}^2\big(\gamma\|\bpv-\rv\|_{L^2}^2 + \|\buv-\vv\|_{L^2}^2\big)\\
&\quad + C \gamma^2 \eps \|(\bpv)_x\|_{L^2}^2+C \eps^2 \intx |\buv|^4 \d x
+  C\eps^4 \|(1+\eps^2 \bpv)^{-2}\|_{L^\infty}^2 \|(\bpv)_x \|_{L^2}^2 \|(\buv)_x\|_{L^2}^2  \\
&\quad + C|\gamma-2|\eps^2 \|(1+\eps^2 \bpv)^{\gamma-3}\|_{L^\infty}\intx |\bpv|^4 \d x. \label{eq:diff_L^2_gronwall}
\end{align*}
Note that, by estimate \eqref{eq:expectation_exp_u_delta} in Proposition \ref{prop:expectation_non_linear},
there exists $C>0$ such that, for any sufficiently small $\eps>0$ and
all $t>0$,
\begin{align*}
&\ex \Big[\int_0^t \frac{\eps}{\gamma}\|(\buv)_x \|_{L^2}^2\d s\Big]
\leq C \big(\eps^{\alpha-2} + \eps^{2\beta-2} \big)< \infty.
\end{align*}
This implies that $\frac{\eps}{\gamma}\|(\buv)_x \|_{L^2}^2$ is locally integrable in time almost surely for sufficiently
small $\eps>0$. The Gronwall's inequality applied to \eqref{eq:diff_L^2_gronwall} gives the desired result.
\end{proof}

\medskip
We now complete the proof of Theorem \ref{thm:finite_time_convergence}.
Fix $T > 0$. Since  \eqref{eq:difference_L^2} is valid for all $0 \leq t \leq T$, we take
the expectation and use the Cauchy-Schwartz inequality to obtain
\begin{align}
&\ex \Big[\supt \big(\gamma\|\bpv(t)-\rv(t)\|_{L^2}^2 + \|\buv(t)-v(t)\|_{L^2}^2\big)\Big]\nonumber \\
&\leq  \bigg(\ex\Big[\exp \Big( \frac{2}{\gamma}\intT \eps\|(\buv)_x \|_{L^2}^2  \d t\Big)\Big]\bigg)^{\frac{1}{2}}
\bigg(\ex \Big[\big(\gamma \|\bpv(0)-\rv(0)\|_{L^2}^2 + \|\buv(0)-v(0)\|_{L^2}^2\big)^2\Big]\bigg)^{\frac{1}{2}}\nonumber\\
&\quad + \bigg(\ex\Big[\exp \Big( \frac{2}{\gamma}\intT \eps\|(\buv)_x \|_{L^2}^2  \d t\Big)\Big]\bigg)^{\frac{1}{2}}\nonumber\\
&\qquad\, \times\bigg(\ex\Big[ \Big(C \eps^2 \intT \intx |\buv|^4 \d x  \d t
+ C\eps^4 \intT \|(1+\eps^2 \bpv)^{-2}\|_{L^\infty}^2 \|(\bpv)_x \|_{L^2}^2 \|(\buv)_x\|_{L^2}^2 \d x  \d t\nonumber\\
& \qquad\qquad\qquad + C |\gamma-2|\eps^2 \intT \|(1+\eps^2 \bpv)^{\gamma-3}\|_{L^\infty} \intx |\bpv|^4 \d x  \d t
+  C\gamma^2 \eps \intT \|(\bpv)_x\|_{L^2}^2  \d t \Big)^2\Big]\bigg)^{\frac{1}{2}}. \\ \label{eq:L^2_diff_exp_control}
\end{align}
We can control the expression on the right-hand side by using the estimates on
the derivation variables obtained in Proposition \ref{prop:expectation_non_linear}. To begin with, since $\alpha>1$ and $\beta> \frac{1}{2}$,
one may take $\eps$ to be sufficiently small such that the choice $\delta =\frac{2 \eps}{\gamma}$ satisfies constraint \eqref{eq:expectation_exp_u_delta}.
Hence, Proposition \ref{prop:expectation_non_linear} implies that, for all sufficiently small $\eps >0$,
\begin{align*}
\ex\Big[\exp \Big( \frac{2}{\gamma}\intT \eps\|(\buv)_x \|_{L^2}^2  \d t\Big)\Big]
\leq C.
\end{align*}
Next, using also Proposition \ref{prop:expectation_non_linear}, we have
\begin{align*}
 \ex \bigg[ \Big( \intT \intx \buv^4 \d x \d t\Big)^2\bigg]
&\leq     \ex \bigg[\Big(\int_0^T \|\buv\|_{L^\infty}^2\|\buv\|_{L^2}^2 \d t \Big)^2\bigg] \\
&\leq \ex \bigg[\Big(\sup_{0 \leq t \leq T}\|\buv\|_{L^2}^2    \int_0^T \|(\buv)_x\|^2_{L^2}  \d t\Big)^2\bigg]\\
&\leq  \bigg(\ex \Big[\Big(\sup_{0 \leq t \leq T}\|\buv\|_{L^2}^2  \Big)^4\Big]\bigg)^{\frac{1}{2}}
\bigg(\ex \Big[ \Big(\int_0^T\|(\buv)_x\|^2_{L^2}  \d t\Big)^4\Big]\bigg)^{\frac{1}{2}}\\
&\leq C   \big(\eps^{\alpha-2} + \eps^{2\beta-2} \big)^4.
\end{align*}
Likewise, Proposition \ref{prop:expectation_non_linear} also gives
\begin{align*}
& \ex \bigg[\Big(\intT \|(1+\eps^2 \bpv)^{-2}\|_{L^\infty}^2 \|(\bpv)_x \|_{L^2}^2 \|(\buv)_x\|_{L^2}^2\d t \Big)^2\bigg] \\
&\leq \ex \bigg[ \Big(\supt \|(1+\eps^2 \bpv)^{-2}\|_{L^\infty}^2    \supt \|(\bpv)_x \|_{L^2}^2    \intT \|(\buv)_x\|_{L^2}^2 \d t\Big)^2 \bigg] \\	
&\leq \bigg(\ex \Big[\Big( \supt \|(1+\eps^2 \bpv)^{-2}\|_{L^\infty}^2\Big)^6\Big] \bigg)^{\frac{1}{3}}
\bigg( \ex \Big[\Big( \supt \|(\bpv)_x \|_{L^2}^2 \Big)^{6}\Big]\bigg)^{\frac{1}{3}}
\bigg(\ex \Big[\Big(\intT \|(\buv)_x \|_{L^2}^2 \d t\Big)^{6}\Big]\bigg)^{\frac{1}{3}}\\
&\leq C \eps^{-4}  \big(\eps^{\alpha-2} + \eps^{2\beta-2} \big)^4.
\end{align*}
By Proposition \ref{prop:expectation_non_linear} again, we have
\begin{align*}
&\ex \bigg[ \Big(\intT  \|(1+\eps^2 \bpv)^{\gamma-3}\|_{L^\infty} \intx |\bpv|^4 \d x \d t\Big)^2 \bigg]\\
&\leq    \ex \bigg[ \Big( \supt \|(1+\eps^2 \bpv)^{\gamma-3}\|_{L^\infty}  \intT  \|\bpv\|_{L^\infty}^2\|\bpv\|_{L^2}^2 \d t\Big)^2\bigg] \\
&\leq  \ex \bigg[  \Big(\supt \|(1+\eps^2 \bpv)^{\gamma-3}\|_{L^\infty} \supt \|\bpv\|_{L^2}^2  \intT  \|(\bpv)_x\|_{L^2}^2   \d t \Big)^2\bigg]\\
&\leq  \bigg(\!\ex \Big[  \supt \|(1+\eps^2 \bpv)^{\gamma-3}\|_{L^\infty}^6 \Big] \bigg)^{\frac{1}{3}}
\!\bigg(\! \ex \Big[ \Big(\supt \| \bpv\|_{L^2}^2 \d x \Big)^{\! 6}\Big] \bigg)^{\frac{1}{3}}
\!\bigg( \!\ex\Big[ \Big(\intT  \|(\bpv)_x\|_{L^2}^2   \d t\Big)^{\!\!6}\Big]\bigg)^{\frac{1}{3}}\\
&\leq C   \big(\eps^{\alpha-2}+ \eps^{2\beta-2} \big)^4.
\end{align*}
Finally, a further application of Proposition \ref{prop:expectation_non_linear}  entails
\begin{align*}
&\ex \bigg[  \Big(\intT \|(\bpv)_x\|_{L^2}^2 \d t\Big)^2 \bigg]
\leq C  \big(\eps^{\alpha-2} +  \eps^{2\beta-2} \big)^2.
\end{align*}
Combining all of the above shows that the right-hand side of \eqref{eq:L^2_diff_exp_control} can be
controlled as stated in \eqref{eq:theorem_1_L^2_diff_exp}, so that the proof of Theorem \ref{thm:finite_time_convergence} is now completed.

\section{Convergence of Invariant Measures} \label{section:convergence_invariant_measures}
In this section, we consider the convergence of invariant measures of system \eqref{eq:nonlinear}
to the corresponding invariant measures of system \eqref{linear} as $\eps \to 0$, which leads
to the complete proof of Theorem \ref{thm:convergence_invariant_measures}.
We restrict to the case that $\gamma=1$, the only case that the existence of invariant measures
for the Navier-Stokes Equations \eqref{eq:NS} (and hence \eqref{eq:nonlinear}) has been known so far.

\subsection{Invariant measures for the fluctuation system}
Let
\begin{equation*}
\bXv = \Big\{ (\bpv, \buv) \in  H^1(0, 1) \times H^1_0(0, 1)\, :\,
\intx \bpv \d x = 0, \,\, 1+\eps^2 \bpv >0 \Big\}
\end{equation*}
be the state space for system \eqref{eq:nonlinear}.
Following \cite{michele}, we equip $\bXv$ with the $L^2 \times L^2$ metric (see Remark \ref{remark:different_norms}).
The Markov semigroup $(\bar{\mathcal{P}}_t^\eps)_{t \geq 0}$ associated to
system \eqref{eq:nonlinear} is a collection of operators
$\bar{\mathcal{P}}_t^\eps : \mathcal{M}_b(\bXv) \to \mathcal{M}_b(\bXv)$ on the set of all bounded Borel measurable functions on $\bXv$ given by
\begin{align*}
\bar{\mathcal{P}}_t^\eps(\phi)(\bp_{0, \eps}, \bu_{0, \eps})
= \ex \big[\phi (\bpv(t, \cdot), \buv(t, \cdot))\big]
\end{align*}
for all $\phi \in  \mathcal{M}_b(\bXv)$ and $(\p_{0,\eps}, u_{0,\eps}) \in \bXv$,
where $(\bpv, \buv)$ is the solution of problem \eqref{ID-1}--\eqref{BC-1} for system \eqref{eq:nonlinear}
with initial data $(\bp_{0, \eps}, \bu_{0, \eps}) \in\bXv$.
A Borel probability measure $\bmv$ on $\bXv$ is said to be an invariant measure of
semigroup $(\bar{\mathcal{P}}_t^\eps)_{t \geq 0}$ if
\begin{equation*}
(\bar{\mathcal{P}}_t^\eps)^* \bmv = \bmv \qquad \text{for all }t \geq 0,
\end{equation*}
where $(\bar{\mathcal{P}}_t^\eps)^* \bmv$ is the push-forward measure of $\bmv$ by $(\bar{\mathcal{P}}_t^\eps)^*$.
Equivalently, the probability measure $\bmv$ satisfies
\begin{equation*}
\int_{\bXv} \bar{\mathcal{P}}_t^\eps(\phi) \d \bmv = \int_{\bXv} \phi \d \bmv
\qquad\,\, \mbox{for all $t \geq 0$ and $\phi \in \mathcal{M}_b(\bXv)$}.
\end{equation*}
An invariant measure represents the law of a statistically stationary solution.

We briefly describe the main result in \cite{michele}, which implies the existence of an invariant probability
measure $\bmv$ for $(\bar{\mathcal{P}}_t^\eps)_{t \geq 0}$.
In \cite{michele},
the Naiver-Stokes equations \eqref{eq:NS} are considered in the phase space
\begin{align*}
\X = \Big\{ (\p, u) \in H^1(0,1) \times H^1_0(0, 1)\,:\, \intx \p \d x =1, \,\,\p >0 \Big\}
\end{align*}
equipped with the $L^2 \times L^2$ norm. For each $\eps>0$,
it has been proved that there exists an invariant measure $\mu_\eps$
for the corresponding Markov semigroup $(\mathcal{P}_t^\eps)_{t \geq 0}$.
Now we can define a bijection $\mathcal{H}_\eps: \bXv \to \X$ by
\begin{equation*}
\mathcal{H}_\eps(\bpv, \buv) = (1+\eps^2 \bpv, \,\eps \buv).
\end{equation*}
Then $\bmv := (\mathcal{H}_\eps^{-1})^* \mu_\eps$ is an invariant measure for
semigroup $(\bar{\mathcal{P}}_t^\eps)_{t \geq 0}$. Indeed, for all $\phi \in \mathcal{M}_b (\bXv)$,
\begin{align*}
\int_{\bXv} \bar{\mathcal{P}}_t^\eps \phi(\bpv, \buv) \d \bmv
&= \int_{\X} \bar{\mathcal{P}}_t^\eps \phi\big(\mathcal{H}_\eps^{-1}(\pv, \uv)\big) \d \mu_\eps
= \int_{\X} \mathcal{P}_t^\eps\big(\phi \circ \mathcal{H}_\eps^{-1}\big) (\pv, \uv) \d \mu_\eps\\
&= \int_{\X} \phi \circ \mathcal{H}_\eps^{-1}(\pv, \uv) \d \mu_\eps =  \int_{\bXv} \phi(\bpv, \buv) \d \bmv.
\end{align*}

\begin{remark} \label{remark:different_norms}
The state space $\X$ considered in \cite{michele}, equipped with the $L^2 \times L^2$ norm, is a non-complete metric space.
Nonetheless, one can establish an \emph{approximate Feller} property for  $\mathcal{P}_t^\eps$, which in turn facilitates the
construction of an invariant measure. However, since the $L^2 \times L^2$ topology and the $H^1 \times H^1_0$ topology generate the same
Borel $\sigma$-algebra on $\X$, the invariance of a Borel probability measure with respect
to $\mathcal{P}_t^\eps$ is not affected by the norm chosen.
\end{remark}

\subsection{Invariant measures for the acoustic system}
The state space for system \eqref{linear} is given in \eqref{eq:state_space_acoustic} by
\begin{align*}
\bX = \Big\{ (r, v) \in  H^1(0, 1) \times H^1_0(0, 1)\,: \,\intx r \d x = 0\Big\}.
\end{align*}
We may equip the space, $\bX$, with either the $L^2 \times L^2$ norm or the $H^1 \times H^1_0$
norm\footnote{For the same reason as in Remark \ref{remark:different_norms},
the choice of the norm does not affect the invariance of a Borel probability measure.},
and we denote them by $\bX_{L^2}$ and $\bX_{H^1}$ respectively when we want to emphasis the choice of the metric.
Similar to the above case, given the existence and uniqueness of solutions to system \eqref{linear},
there is a well-defined semigroup $(\mathcal{Q}_t^{\eps})_{t \geq 0}$ associated to \eqref{linear}  in $\bX$
for all $\eps>0$.
More specifically, for all $t \geq 0$,
map $\mathcal{Q}_t^{\eps}: \mathcal{M}_b(\bX) \to \mathcal{M}_b(\bX)$
is given by
\begin{align*}
\mathcal{Q}_t^{\eps}(\phi)(r_{0, \eps}, v_{0, \eps}) = \ex \phi (\rv(t, \cdot), \vv(t, \cdot))
\end{align*}
for all $\phi \in  \mathcal{M}_b(\bX)$ and $(r_{0, \eps}, v_{0, \eps}) \in \bX$, where
$(\rv, \vv)$ is the solution of problem \eqref{ID-2}--\eqref{BC-2} for system \eqref{linear}
with initial data $(r_{0, \eps}, v_{0, \eps})$. Moreover, there is a unique invariant
measure $\omega_\eps$ for $(\mathcal{Q}_t^{\eps})_{t \geq 0}$ (see Corollary \ref{cor_exp_contraction_invar_measure} below).

In order to compare the invariant measures between systems \eqref{eq:nonlinear} and \eqref{linear},
we want to embed $\bXv$ into $\bX$.
Note that $\bar{\mathcal{X}}_{\eps_2} \subset \bar{\mathcal{X}}_{\eps_1}$ for $0 < \eps_1 < \eps_2$.
In fact, we have the following relation:
\begin{align*}
\bX = \lim_{\eps \to 0} \bXv := \bigcup_{\eps >0} \bXv.
\end{align*}
As a result, for any probability measure $\nu$ on $\bX$ for $\eps >0$,
we may extend it to be a probability measure $\tilde{\nu}$ on $\bX$
by defining for all Borel measurable set $A \subseteq \bX$,
\begin{align}\label{eq:measure_extension}
\tilde{\nu}(A) = \nu (A \cap \bXv).
\end{align}
In the sequel, we omit the tilde sign and abuse the notation $\nu$ to
mean its natural extension $\tilde{\nu}$.

\subsection{Contraction property of the acoustic system semigroup}

In this section, we prove a contraction property of semigroup $(\mathcal{Q}_t^{\eps})_{t\geq 0}$,
which will be crucial to translate the finite-time convergence of solutions to the convergence of invariant measures.

Since the solutions are not known to be pathwise contractive in the $L^2 \times L^2$ norm,
we define an $\eps$-dependent norm that coincides with the $L^2 \times L^2$ norm
in the limit $\eps \to 0$ and has the pathwise contraction property.
For all $\eps \in (0,1]$, define the $\eps$-norm $\| \cdot\|_\eps$ on $\bX$ by
\begin{align} \label{eq:def_epsilon_norm}
\| (r, v)\|_\eps^2 = \|(r,\,v)\|_{L^2}^2 + \eps^2 \|(r_x,\,v_x)\|_{L^2}^2.
\end{align}
We now define the Wasserstein metrics for probability measures on $\bX$
by using the metrics on $\bX$.
For a metric $\th$ on $\bX$,
let\footnote{Here, we denote the set of all Borel probability measures on $\bX$ by $\mathfrak{P}(\bX)$.}
\begin{equation*}
 \mathfrak{P}_\th = \Big\{ \bar\eta \in \mathfrak{P}(\bX)\,:\,
  \int_{\bX} \th((0, 0), (r, v)) \d \bar\eta (r, v) < \infty \Big\}.
\end{equation*}
be the collection of Borel probability measures on $\bX$ with finite metric moment.
Then the Wasserstein metric $\mathcal{W}_\th$ associating with $\th$ is defined
for all $\eta_1, \eta_2 \in \mathfrak{P}_\th$ as
\begin{align} \label{def_wasserstein}
\mathcal{W}_{\th} (\eta_1, \eta_2) = \inf_{\bar\eta \in \mathcal{C}(\eta_1, \eta_2)}
\int_{\bX \times \bX} \th\big((r_1, v_1), (r_2, v_2)\big) \d \bar\eta((r_1,v_1), (r_2, v_2)),
\end{align}
where the infimum is taken over the set, $\mathcal{C}(\eta_1, \eta_2)$, of all couplings of $\eta_1$ and $\eta_2$,
{\it i.e.}, all measures on $\mathcal{X} \times \mathcal{X}$ with marginals $\eta_1$ and $\eta_2$.
Note that $\mathcal{W}_\th$ turns $\mathfrak{P}_\th$ into a metric space.
If, in addition, $\th$ is complete, then so is $\mathcal{W}_\th$; see for instance \cite{Villani09}.
We denote by $\mathcal{W}_{L^2}, \mathcal{W}_{H^1}$, and $\mathcal{W}_\eps$
the Wasserstein metrics corresponding to the $L^2 \times L^2$ metric, the $H^1 \times H^1_0$ metric,
and the $\eps$-norm \eqref{eq:def_epsilon_norm}, respectively.
We also write $\mathfrak{P}_{L^2}$, $\mathfrak{P}_{H^1}$, and $\mathfrak{P}_\eps$
for the corresponding collections of probability measures.
Moreover, we remark that the infimum in \eqref{def_wasserstein} can be attained
for all the three metrics $\mathcal{W}_{L^2}, \mathcal{W}_{H^1}$,
and $\mathcal{W}_\eps$ ({\it cf.}  \cite{Villani09}*{Theorem 4.1}).
Moreover, it follows from the Monge–Kantorovich duality that
\begin{align} \label{def_monge}
\mathcal{W}_{\th} (\eta_1, \eta_2)
= \sup_{\textrm{Lip}_{\th} (\phi)\leq 1 } \left| \int_{\bX} \phi \d \eta_1 - \int_{\bX} \phi \d\eta_2 \right|,
\end{align}
where the supremum is taken over the set of all Lipschitz functions $\phi: \bX \to \mathbb{R}$
with Lipschitz constants less than $1$ when $\bX$ is equipped with metric $\th$.

We can establish the following contraction property for semigroup $\mathcal{Q}_t^{\eps}$ on $\bX_{H^1}$,
as well as on $\bX$ equipped with the $\eps$-norm, which may be
seen as a simple version of  \cite{hairer2008spectral}*{Theorem 2.5}.
The main ingredient is the exponential contraction properties of solutions
expressed in Corollary \ref{cor_linear_exponential_approach}.

\begin{proposition} \label{prop:exp_contraction}
Let $\eps \in (0, 1]$. Then, for all $t \geq 0$
and for probability measures $\eta_1$ and $\eta_2$ on $\bX$,
\begin{align}
&\mathcal{W}_{H^1} ((\mathcal{Q}_t^{\eps})^* \eta_1, (\mathcal{Q}_t^{\eps})^* \eta_2)
\leq  2 \sqrt{2}\e^{-\frac{t}{8}} \mathcal{W}_{H^1} (\eta_1, \eta_2),\label{eq:contra1}\\
&\mathcal{W}_{\eps} ((\mathcal{Q}_t^{\eps})^* \eta_1, (\mathcal{Q}_t^{\eps})^* \eta_2)
\leq  \sqrt{6}\e^{-\frac{t}{4}} \mathcal{W}_{\eps} (\eta_1, \eta_2),\label{eq:contra2}
\end{align}
where $(\mathcal{Q}_t^{\eps})^*\bar\eta $ is the push-forward measure of $\bar\eta$ by $\mathcal{Q}_t^\eps$.
\end{proposition}

\begin{proof}
We first show that the conclusion is true when $\eta_1$ and $\eta_2$ are the Dirac measures.
Indeed, let $g_1, g_2 \in \bX$. Then, by the Monge–Kantorovich duality \eqref{def_monge}, we have
\begin{align*}
\mathcal{W}_{H^1} ((\mathcal{Q}_t^{\eps})^* \delta_{g_1}, (\mathcal{Q}_t^{\eps})^* \delta_{g_2})
= \sup_{\textrm{Lip}_{H^1} (\phi)\leq 1 }
\left|\mathcal{Q}_t^{\eps} \phi (g_1)- \mathcal{Q}_t^{\eps} \phi (g_2)\right|.
\end{align*}
Denote by $(\rv(t; g), \vv(t, g))$ the solution of system \eqref{linear}
at time $t \geq 0$ stating at the initial condition: $(r_{0, \eps}, v_{0, \eps})= g$.
Taking any $\phi \in C(\bX_{H_1})$ with $\textrm{Lip}_{H_1} (\phi)\leq 1$, we have
\begin{align*}
\left|\mathcal{Q}_t^{\eps} \phi (g_1)- \mathcal{Q}_t^{\eps} \phi (g_2)\right|
& = \left|\mathbb{E} \big[\phi (\rv(t; g_1), \vv(t, g_1))
  -\phi (\rv(t; g_2), \vv(t, g_2))\big]\right|\\
&\leq \mathbb{E} \Big[ \left(\|\rv(t; g_1)-\rv(t; g_2)\|^2_{H^1}
  + \|\vv(t; g_1)-\vv(t; g_2)\|^2_{H^1} \right)^{\frac{1}{2}}\Big]\\
&\leq 2 \sqrt{2}\e^{-\frac{t}{8}} \|g_1-g_2\|_{H^1 \times H^1_0}\\
&=  2 \sqrt{2}\e^{-\frac{t}{8}}  \mathcal{W}_{H^1} (\delta_{g_1}, \delta_{g_2}).
\end{align*}
The second inequality is due to \eqref{eq:exp_decay_H_1} in Corollary \ref{cor_linear_exponential_approach}.

Next, we prove the general case.
Let $\eta_1$ and $\eta_2$ be two probability measures on the $\bX$.
Take $ \bar\eta \in \mathcal{C}(\eta_1, \eta_2)$ such that the infimum in \eqref{def_wasserstein} is attained.
Then
\begin{align*}
\mathcal{W}_{H^1} ((\mathcal{Q}_t^{\eps})^*\eta_1, (\mathcal{Q}_t^{\eps})^*\eta_2)
&\leq \int_{\mathcal{X} \times \mathcal{X}}
\mathcal{W}_{H^1} ((\mathcal{Q}_t^{\eps})^*\delta_{g_1},
 (\mathcal{Q}_t^{\eps})^*\delta_{g_2}) \d \bar\eta(g_1,g_2)\\
&\leq  2 \sqrt{2}\e^{-\frac{t}{8}} \int_{\mathcal{X} \times \mathcal{X}}
\mathcal{W}_{H^1} (\delta_{g_1}, \delta_{g_2}) \d \bar\eta(g_1, g_2)\\
&=2 \sqrt{2}\e^{-\frac{t}{8}} \mathcal{W}_{H^1} (\eta_1, \eta_2).
\end{align*}
This proves \eqref{eq:contra1}.
The proof of \eqref{eq:contra2} is the same, except
\eqref{eq:exp_decay_epsilon_H1} is used in place of \eqref{eq:exp_decay_H_1}.
\end{proof}

The contraction property implies the existence and uniqueness of the invariant probability measure
for system \eqref{linear}.

\begin{corollary} \label{cor_exp_contraction_invar_measure}
The semigroup $(\mathcal{Q}_t^{\eps})_{t\geq 0}$ admits the unique invariant measure $\omega_\eps$
on $\bX_{H^1}$ such that $\omega_\eps \in \mathfrak{P}_{H^1}$ and, for all $\lambda \in \mathfrak{P}_{H^1}$,
\begin{equation} \label{eq:contraction_omega}
\mathcal{W}_{H^1} ((\mathcal{Q}_t^{\eps})^* \lambda, \omega_\eps)
\leq  2 \sqrt{2}\e^{-\frac{t}{8}} \mathcal{W}_{H^1} (\lambda, \omega_\eps)
\qquad\,\, \mbox{for every $t\geq 0$}.
\end{equation}
\end{corollary}

\begin{proof}
Note that, by Proposition \ref{prop:exp_contraction},
map $(\mathcal{Q}_t^{\eps})^*$ is a contraction on $\mathfrak{P}_{H^1}$
for all sufficiently large $t \geq 0$.
Since $H^1 \times H^1_0$ is a complete norm on $\bX$,
$\mathfrak{P}_{H^1}$ is a complete space.
By the contraction mapping principle, for all such $t$,
there exists a probability measure $\omega_\eps^t \in \mathfrak{P}_{H^1}$
such that $(\mathcal{Q}_t^{\eps})^* \omega_\eps^t = \omega_\eps^t$.
The time-continuity of $(\mathcal{Q}_t^{\eps})$ implies
that $\omega_\eps^t = \omega_\eps^s$ for all these sufficiently large $t$ and $s$.
By the semigroup property,  $\omega_\eps = \omega_\eps^t$ is invariant for all $t \geq 0$.
The attraction property \eqref{eq:contraction_omega} is a consequence of Proposition \ref{prop:exp_contraction}.

To complete the proof, we still have to show that the uniqueness holds in the entire space $\mathfrak{P}(\bX)$,
not just in $\mathfrak{P}_{H^1}$.
To see this, let $\bar{\omega}_\eps \in \mathfrak{P}(\bX)$ be an invariant measure
for $(\mathcal{Q}_t^{\eps})_{t\geq 0}$.
Note that, for all $\delta>0$, there exists $M_\delta>0$ such that
\begin{equation*}
 m_\eps^\delta:= \bar{\omega}_\eps\big(\big\{(r,v) \in \bX: \|r_x\|_{L^2}+\|v_x\|_{L^2} > M_\delta \big\}\big) < \delta.
\end{equation*}
Define $\bar{\omega}_\eps^\delta \in \mathfrak{P}_{H^1}$ by
\begin{equation*}
\bar{\omega}_\eps^\delta (A)
= \frac{1}{1-m_\eps^\delta}\int_A \mathbbm{1}_{\{\|r_x\|_{L^2}+\|v_x\|_{L^2} \leq M_\delta\}} \d \bar{\omega}_\eps
\end{equation*}
for all Borel measurable set $A \subseteq \bX$.
On one hand, since $ \bar{\omega}_\eps^\delta$ converges to $\bar{\omega}_\eps$
setwise as $\delta \to 0$, $(\mathcal{Q}_t^{\eps})^*\bar{\omega}_\eps^\delta$
converges to $(\mathcal{Q}_t^{\eps})^*\bar{\omega}_\eps = \bar{\omega}_\eps$
setwise and hence weakly as $\delta \to 0$.
On the other hand, since $\bar{\omega}_\eps^\delta \in \mathfrak{P}_{H^1}$,
\eqref{eq:contraction_omega} implies that
$(\mathcal{Q}_t^{\eps})^*\bar{\omega}_\eps^\delta$ converges
weakly to $\omega$ as $t \to \infty$.
This yields $\bar{\omega}_\eps=\omega_\eps$ by exchanging the order of the two limits,
which is valid since the setwise (and hence weak) convergence of $(\mathcal{Q}_t^{\eps})^*\bar{\omega}_\eps^\delta$
to $\bar{\omega}_\eps$ as $\delta \to 0$ is uniform in $t\geq 0$. This completes the proof.
\end{proof}

We can also deduce further estimates for the invariant measure $\omega_\eps$,
which can be seen as moment bounds for the corresponding statistically stationary solutions.

\begin{proposition} \label{prop:est_invar_measure_linear}
Let $\eps \in (0, 1]$. Let $\omega_\eps$ be the invariant measure of $(\mathcal{Q}_t^{\eps})$.
Then
\begin{align} \label{eq:invms_derivative}
\int_{\bX} \|(r_x,\, v_x)\|_{L^2}^2 \d \omega_\eps \leq \|\sigma\|_{L^\infty}^2 \eps^{2(\beta-1)}.
\end{align}
Moreover, for all $p \geq 1$, there exists $C_p>0$, independent of $\eps$, such that
\begin{align}\label{eq:invmsrpoly}
\int_{\bX} \mathcal{G}(r, v)^p \d \omega_\eps \leq C_p \eps^{2(\beta-1)p} \|\sigma\|_{L^\infty}^{2p}.
\end{align}
Furthermore, for all $\eta\in (0, \,(8\eps^{2(\beta-1)} \|\sigma\|_{L^\infty}^2)^{-1})$,
there exists $C_\eta>0$, independent of $\eps$, such that
\begin{equation}\label{eq:invmsrexp}
\int_{\bX}\exp( \eta \mathcal{G}(r, v)) \d \omega_\eps
\leq C_\eta \exp(2\eta\eps^{2(\beta-1)}).
\end{equation}
\end{proposition}

\begin{proof}
First, we integrate \eqref{eq:energy_balance_linear} with respect to measure $\omega_\eps$,
which implies that, for all $t \geq 0$,
\begin{equation}
    \int_{\bX} \mathcal{Q}_t^\eps \mathcal{G}(r,v) \d \omega_\eps
    + \frac{1}{2}\intt \int_{\bX} \mathcal{Q}_s^\eps (\|(r_x,\,v_x)\|_{L^2}^2 ) \d \omega_\eps  \d s
    = \int_{\bX} \mathcal{G} (r,v)\d \omega_\eps + \frac{1}{2}\eps^{2(\beta-1)}t \intx \sigma^2 \d x. \label{eq:energy_balance_linear_invar_meas}
\end{equation}
Using the invariance of $\omega_\eps$ and the fact that $\omega_\eps \in \mathfrak{P}_{H^1}$, we have
\begin{align*}
   \int_{\bX} \mathcal{Q}_t^\eps \mathcal{G}(r,v) \d \omega_\eps= \int_{\bX} \mathcal{G} (r,v)\d \omega_\eps < \infty.
\end{align*}
Therefore, \eqref{eq:energy_balance_linear_invar_meas} implies
\begin{align*}
    \frac{1}{2}t\int_{\bX} \|(r_x, \,v_x)\|_{L^2}^2 \d \omega_\eps = \frac{1}{2}\eps^{2(\beta-1)}t \intx \sigma^2 \d x,
\end{align*}
from which \eqref{eq:invms_derivative} follows.
Next, let $\delta_{(0, 0)}$ be the delta measure on $\bX$ centred at $(0, 0) \in \bX$.
By Proposition \ref{prop:large_time_linear}, we see that, for all $p \geq 1$,
there exists $C_p>0$, independent of $\eps$, such that
\begin{equation*}
\lim_{t\to\infty}\int_{\bX} \mathcal{G}(r, v)^p \d (\mathcal{Q}_t^{\eps})^* \delta_{(0, 0)}
= C_p \eps^{2(\beta-1)p} \|\sigma\|_{L^\infty}^{2p}.
\end{equation*}
Likewise,  for all $\eta\in (0, \,(8\eps^{2(\beta-1)} \|\sigma\|_{L^\infty}^2)^{-1})$,
there exists $C_\eta>0$, independent of $\eps$, such that
\begin{equation*}
\lim_{t\to\infty}\int_{\bX} \exp(\eta \mathcal{G}(r, v) \d (\mathcal{Q}_t^{\eps})^* \delta_{(0, 0)}
= C_\eta \exp\big(2\eta\eps^{2(\beta-1)}\big).
\end{equation*}
By Corollary \ref{cor_exp_contraction_invar_measure}, we have
\begin{align*}
\lim_{t\to\infty}\mathcal{W}_{H^1} ((\mathcal{Q}_t^{\eps})^* \delta_{(0, 0)}, \omega_\eps) = 0.
\end{align*}
This can be rephrased in terms of the Monge–Kantorovich duality \eqref{def_monge}.
For all globally Lipschitz function $f: \bX \to \mathbb{R}$,
\begin{align}\label{eq:liminvmsr}
\lim_{t\to\infty}\int_{\bX} f \d (\mathcal{Q}_t^{\eps})^* \delta_{(0, 0)}
= \int_{\bX} f \d \omega_\eps .
\end{align}
Let $p \geq 1$ and $N >0$. Choosing $f(r, v) = \min\{\mathcal{G}(r, v)^p , N\}$ in the above gives
\begin{align*}
\int_{\bX} \min\{\mathcal{G}(r, v)^p , N\} \d \omega_\eps
&= \lim_{t \to \infty} \int_{\bX} \min\{\mathcal{G}(r, v)^p , N\} \d (\mathcal{Q}_t^{\eps})^* \delta_{(0, 0)} \\
&\leq \lim_{t \to \infty} \int_{\bX} \mathcal{G}(r, v)^p \d (\mathcal{Q}_t^{\eps})^* \delta_{(0, 0)}
= C_p \eps^{2(\beta-1)p} \|\sigma\|_{L^\infty}^{2p}.
\end{align*}
Now, letting $N \to \infty$ gives the first estimate \eqref{eq:invmsrpoly}.
Similarly, let $\eta \in (0, (8\eps^{2(\beta-1)} \|\sigma\|_{L^\infty}^2)^{-1})$ and $M >0$.
Choosing $f(r, v) = \min\{\exp( \eta \mathcal{G}(r, v)) , M\}$ in \eqref{eq:liminvmsr} gives
\begin{align*}
\int_{\bX} \min\{\exp( \eta \mathcal{G}(r, v)) , M\} \d \omega_\eps
&= \lim_{t \to \infty} \int_{\bX}\min\{\exp( \eta \mathcal{G}(r, v)), M\} \d (\mathcal{Q}_t^{\eps})^*\delta_{(0, 0)} \\
&\leq \lim_{t \to \infty} \int_{\bX} \exp( \eta \mathcal{G}(r, v))\d (\mathcal{Q}_t^{\eps})^* \delta_{(0, 0)}
= C_\eta \exp\big(2\eta\eps^{2(\beta-1)}\big).
\end{align*}
Now, letting $M \to \infty$ gives the second estimate \eqref{eq:invmsrexp}.
\end{proof}

\subsection{Estimates of the invariant measures} \label{subsect:est_invar_meas}
In order to apply the finite-time convergence result to study the convergence of invariant measures,
we have to estimate the statistically stationary solutions distributed as invariant measures.
In particular, we want to show that they satisfy \eqref{eq:assumption_moment_initial_poly},
which is the main goal of this section.

Since there is a natural correspondence between the invariant measures of systems \eqref{eq:nonlinear} and \eqref{eq:NS},
we can confine our analysis to the latter system.
Recall that the  energy balance \eqref{eq:modified_energy_balance} holds for system \eqref{eq:NS}:
\begin{align}
&\d \cE(\p_\eps, u_\eps)
+\frac{1}{2} \Big(\intx|(\uv)_x|^2 \d x  + \frac{1}{\eps}\intx \pv^{\gamma-3}|(\pv)_x|^2 \d x\Big) \d t \\
&= \frac{1}{2}\eps^{2\beta}\intx\pv \sigma^2 \d x \d t
+ \eps^\beta \intx \big(\pv \uv + \frac{1}{2} \frac{(\pv)_x}{\pv}\big) \sigma \d x \d W_t.
\label{eq:modified_energy_balance_ch3}
\end{align}
To simplify the notation, we write that, for all $(\pv, \uv) \in \X$,
\begin{align*}
&\mathcal{K}(\pv, \uv)= \frac{1}{2} \|(\uv)_x\|_{L^2}^2 + \frac{1}{2 \eps^2}  \|(\log \pv)_x\|_{L^2}^2,\\
&\mathcal{M}(\pv, \uv)= \eps^{\beta} \intx \big(\pv \uv + \frac{1}{2} \frac{(\pv)_x}{\pv}\big)\sigma \d x,\\
&\mathcal{R}(\pv, \uv)=\frac{\eps^{2\beta}}{2} \intx \pv \sigma^2 \d x,
\end{align*}
so that the above equality reads
\begin{equation*}
\d\cE(\pv, \uv) +\mathcal{K}(\pv, \uv)\d t = \mathcal{M}(\pv, \uv) \d W_t +\mathcal{R}(\pv, \uv)\d t.
\end{equation*}
Note here that the terms $ \mathcal{M}(\pv, \uv) \d W_t$ and $\mathcal{R}(\pv, \uv) \d t$
should be interpreted as
\begin{align*}
&\mathcal{M}(\pv, \uv) \d W_t
 = \sum_{k=1}^\infty \eps^{\beta} \intx \big(\pv \uv + \frac{1}{2} \frac{(\pv)_x}{\pv}\big)\sigma_k \d x \d W^k_t, \\
&\mathcal{R}(\pv, \uv) \d t = \sum_{k=1}^\infty \frac{\eps^{2\beta}}{2} \intx \pv \sigma_k^2 \d x \d t.
\end{align*}
We also define
\begin{align*}
\mathcal{M}(\pv, \uv)^2
:= \sum_{k=1}^\infty \Big( \eps^{\beta} \intx \big(\p u + \frac{1}{2} \frac{\p_x}{\p}\big)\sigma_k \d x \Big)^2
\end{align*}
for later use.

As a first step of the estimates, we show that $\cE(\pv, \uv)$ can be controlled
by $\mathcal{K}(\pv, \uv)$, a distinctive feature that is true only for the isothermal pressure law
and has been observed and used in \cites{Hoff99, Hoff00} to study the long-time behavior
of the solutions in the deterministic case.

\begin{proposition} Let $\eps \in (0, 1]$. Then, for all $(\p, u)\in \X$,
\begin{align}\label{eq:bddEK1}
\mathcal{E}(\p, u) \leq \big(\, \frac{3}{2}+2\|\p\|_{L^\infty}^2 + \eps^2 \|\p^{-1}\|_{L^\infty} \big)\mathcal{K}(\p, u),
\end{align}
which implies
\begin{align}\label{eq:bddEK2}
\mathcal{E}(\p, u) &\leq \big(\frac{3}{2}+2\exp(2 \eps\sqrt{2\mathcal{K}(\p, u)})
  + \eps^2 \exp (\eps \sqrt{2\mathcal{K}(\p, u)}) \big)\mathcal{K}(\p, u).
\end{align}
\end{proposition}

\begin{proof}
It suffices to control each term in the definition of $\cE(\p, u)$
in \eqref{eq:def_modified energy} separately.
The first two terms are easily bounded by
$$
 \frac{1}{2}\intx \tp \tu^2 \d x \leq \frac{1}{2}\|u\|_{L^\infty}^2 \intx \tp \d x \leq\frac{1}{2} \|u_x\|_{L^2}^2 \leq \mathcal{K}(\p, u),
$$
and
\begin{align*}
\intx \frac{1}{\eps^2}(\tp \log \tp-p+1)\d x
&\leq \frac{1}{\eps^2} \intx \tp (\tp-1) \d x
  = \frac{1}{\eps^2} \intx (\tp^2-1) \d x \leq \frac{1}{\eps^2} \intx \tp_x^2 \d x \\
& \leq \frac{1}{\eps^2}\|\tp\|_{L^\infty}^2 \|(\log \tp)_x\|_{L^2}^2 \leq 2 \|\tp\|_{L^\infty}^2 \mathcal{K}(\p, u).
\end{align*}
Moreover, we have
\begin{align*}
&\frac{1}{4} \intx \frac{\tp_x^2}{\tp^3} \d x
 \leq  \frac{\eps^2}{2}\|\tp^{-1}\|_{L^\infty} \, \frac{1}{2\eps^2}\|(\log \tp)_x\|_{L^2}^2
 \leq \frac{\eps^2}{2}\|\tp^{-1}\|_{L^\infty}  \mathcal{K}(\p, u),\\[0.5mm]
&\frac{1}{2} \intx \frac{\tp_x \tu}{\tp}
\leq \frac{1}{4} \intx \tp \tu^2 \d x + \frac{1}{4} \intx \frac{\tp_x^2}{\tp^3} \d x
\leq \frac{1}{2} \mathcal{K}(\p, u)+\frac{\eps^2}{2}\|\tp^{-1}\|_{L^\infty}  \mathcal{K}(\p, u).
\end{align*}
Combining the above, we obtain the first estimate \eqref{eq:bddEK1}.

The second estimate \eqref{eq:bddEK2} follows from the first one and the fact that
\begin{align*}
\|\log \p\|_{L^\infty} \leq \|(\log \p)_x\|_{L^2} \leq \eps^{-1}\sqrt{2\mathcal{K}(\p, u)}
\end{align*}
for all $\p \in H^1$ with mass equal to 1.
Indeed, there exists $x_0 \in [0, 1]$ such that $\p (x_0) = 1$ so that, for all $x \in [0, 1]$,
\begin{align*}
|\log \p(x)| =| \log \p(x)-\log\p(x_0) |= \Big|\int_{x_0}^x \big(\log \p(y)\big)_y \d y \Big|
\leq \|(\log \p)_y\|_{L^2}.
\end{align*}
\end{proof}

In light of the above proposition, we define $h: [0, \infty) \to [0, \infty)$ given by
\begin{equation}\label{eq:def_h}
h(x): =\big(\frac{3}{2}+2\exp (2 \eps\sqrt{2x} )+ \eps^2 \exp(\eps \sqrt{2x}) \big)x.
\end{equation}
Note that $h$ is strictly increasing, bijective and convex.
Using \eqref{eq:modified_energy_balance_ch3}, we can obtain a first bound for the invariant measures.
The computation here is similar to the one in the incompressible case; see \cite{Flandoli08}.

\begin{proposition}
Let $F: [0, \infty) \to [0, \infty)$ be of class $C^1$ and increasing.
Let $\mu_\eps$ be an invariant measure of system \eqref{eq:NS} defined on the state space $\X$. Then
\begin{equation}\label{eq:invar_meas_est_F}
 \int_{\X} F(\cE(\p, u)) \mathcal{K}(\p, u) \:\d\mu_\eps
 \leq \int_{\X} F(\cE(\p, u)) \mathcal{R}(\p, u)\d\mu_\eps
   + \frac{1}{2} \int_{\X} F'(\cE(\p, u)) (\mathcal{M}(\p, u))^2 \d\mu_\eps,
\end{equation}
where the last integral is to be understood as
\begin{equation*}
\int_{\X} F'(\cE(\p, u)) (\mathcal{M}(\p, u))^2 \d\mu_\eps
= \sum_{k =1}^\infty \int_{\X} F'(\cE(\p, u)) \eps^{2\beta}
\Big(\intx \big(\p u + \frac{1}{2} \frac{\p_x}{\p}\big)\sigma_k \d x \Big)^2 \d\mu_\eps.
\end{equation*}
\end{proposition}

\begin{proof}
Note that, for each $\delta>0$, there exists $N_\delta>0$ such that
\begin{align*}
q_\eps^\delta := \mu_\eps \left(\{(\p, u) \in \X\, : \,\cE(\p, u) > N_\delta \}\right) < \delta.
\end{align*}
Define a probability measure $\mu_\eps^\delta$ on $\X$ by
\begin{equation*}
\mu_\eps^\delta (A) = \frac{1}{1-q_\eps^\delta}\int_A  \mathbbm{1}_{\{\cE(\p, u) \leq N_\delta\} } \d \mu_\eps
\end{equation*}
for all Borel measurable set $A \subseteq \X$.
Let $(\p_{0, \eps}^\delta, u_{0, \eps}^\delta)$ be an $\mathcal{F}_0$-measurable random variable
distributed as $\mu_\eps^\delta$.
Let $(\pv^\delta, \uv^\delta)$ be a solution of the fluctuation
system \eqref{eq:nonlinear} with initial data $(\p_{0, \eps}^\delta, u_{0, \eps}^\delta)$.

Define the function $\tilde{F}: [0, \infty) \to [0, \infty)$ by
\begin{align*}
\tilde{F}(r) = \int_0^r F(y) \d y
\end{align*}
so that $\tilde{F}' = F$. By It\^{o}' s formula, we have
\begin{align*}
\d \tilde{F}(\cE(\pv^\delta, \uv^\delta)
= F( \cE (\pv^\delta, \uv^\delta)) \d\cE(\pv^\delta, \uv^\delta)
+ \frac{1}{2} F'(\cE(\pv^\delta, \uv^\delta)) \d \langle \cE(\pv^\delta, \uv^\delta) , \cE(\pv^\delta, \uv^\delta) \rangle,
\end{align*}
so that
\begin{align*}
& \d \tilde{F}(\cE(\pv^\delta, \uv^\delta) )  + F(\cE(\pv^\delta, \uv^\delta))K(\pv^\delta, \uv^\delta)  \d t \\
&= F(\cE(\pv^\delta, \uv^\delta))M(\pv^\delta, \uv^\delta)  \d W_t
  + F(\cE(\pv^\delta, \uv^\delta))R(\pv^\delta, \uv^\delta)  \d t \\
&\quad+ \frac{1}{2} F'(\cE(\pv^\delta, \uv^\delta))\big(M(\pv^\delta, \uv^\delta) \big)^2 \d t,
\end{align*}
where
\begin{align*}
F'(\cE(\pv^\delta, \uv^\delta))\big(M(\pv^\delta, \uv^\delta) \big)^2 \d t
= \sum_{k=1}^\infty F'(\cE(\pv^\delta, \uv^\delta))
  \Big(  \eps^{\beta} \intx \big(\pv^\delta \uv^\delta + \frac{1}{2} \frac{(\pv)_x^\delta}{\pv^\delta}\big)
    \sigma_k \d x \Big)^2\d t.
\end{align*}
Thus, for all $T>0$, we obtain
\begin{align*}
& \ex \big[ \tilde{F}(\cE(\pv^\delta(T, \cdot), \uv^\delta(T, \cdot)) )\big]
+ \ex \Big[ \intT F(\cE(\pv^\delta(t, \cdot), \uv^\delta(t, \cdot)))K(\pv^\delta(t, \cdot), \uv^\delta(T, \cdot)) \d t \Big]\\
&= \ex \big[ \tilde{F}(\cE(\p_{0, \eps}^\delta, u_{0, \eps}^\delta)) \big]
  + \ex \Big[ \intT F(\cE(\pv^\delta(t, \cdot), \uv^\delta(t, \cdot)))R(\pv^\delta(t, \cdot), \uv^\delta(t, \cdot))\d t \Big]\\
&\quad + \frac{1}{2}\ex \Big[ \intT F'(\cE(\pv^\delta(t, \cdot), \uv^\delta(t, \cdot)))
  \big(M(\pv^\delta(t, \cdot), \uv^\delta(t, \cdot))\big)^2 \d t\Big].
\end{align*}
Therefore, for all $T, P >0$, we have
\begin{align*}
&\int_{\mathcal{X}} \min\{F\left(\cE(\p, u)\right)\mathcal{K}(\p, u) , P\}\d \mu_\eps\\
&=  \frac{1}{T}\intT \int_{\mathcal{X}}
  \mathcal{P}_t(\min\{F(\cE(\p, u))\mathcal{K}(\p, u), P\})\d \mu_\eps \d t\\
&=  \frac{1}{T}\intT \int_{\mathcal{X}}
\mathcal{P}_t (\min\{F(\cE(\p, u))\mathcal{K}(\p, u), P\})
  \mathbbm{1}_{\{\cE(\p, u) \leq N_\delta\} }\d \mu_\eps \d t\\
& \quad +\frac{1}{T}\intT \int_{\mathcal{X}}
  \mathcal{P}_t (\min\{F(\cE(\p, u))\mathcal{K}(\p, u), P\})
  \mathbbm{1}_{\{ \cE(\p, u) > N_\delta\} }\d \mu_\eps \d t\\
&\leq \frac{(1-q_\eps^\delta)}{T}\intT
  \ex\big[ \min\{F(\cE(\pv^\delta(t, \cdot), \uv^\delta(t, \cdot )))K(\pv^\delta(t, \cdot), \uv^\delta(t, \cdot)) , P\} \big] \d t
    + P \delta\\
&\leq \frac{(1-q_\eps^\delta)}{T} \ex \big[ \tilde{F}(\cE(\p_{0, \eps}^\delta, u_{0, \eps}^\delta)) \big]
 +  \frac{(1-q_\eps^\delta)}{T}\ex\Big[\intT F(\cE(\pv^\delta(t, \cdot), \uv^\delta(t, \cdot)))
    R(\pv^\delta(t, \cdot), \uv^\delta(t, \cdot)) \d t \Big]\\
&\quad + \frac{(1-q_\eps^\delta)}{2T} \ex \Big[ \intT F'(\cE(\pv^\delta(t, \cdot ), \uv^\delta(t, \cdot)))
   \big(M(\pv^\delta(t, \cdot), \uv^\delta(t, \cdot))\big)^2 \d t\Big]+ P \delta\\
&\leq \frac{(1-q_\eps^\delta) \tilde{F}(N_\delta)}{T}
  + \frac{1}{T} \intT \int_{\mathcal{X}}  \mathcal{P}_t (F(\cE(\p, u))\mathcal{R}(\p, u))\chi_{\cE(\p, u) \leq N_\delta}\d \mu_\eps \d t \\
&\quad + \frac{1}{2T}\intT  \int_{\mathcal{X}} \mathcal{P}_t(F'(\cE(\p, u)) (M(\tp, u))^2)
   \chi_{\cE(\p, u) \leq N_\delta}\d \mu_\eps  \d t+ P \delta\\
&\leq \frac{(1-q_\eps^\delta) \tilde{F}(N_\delta)}{T}
  + \int_{\mathcal{X}}   F(\cE(\p, u))\mathcal{R}(\p, u) \d \mu_\eps
   + \frac{1}{2}\int_{\mathcal{X}} F'(\cE(\p, u))(M(\tp, u))^2\d \mu_\eps+P\delta.
\end{align*}
The result follows by letting $T \to \infty$, then $\delta \to 0$, and finally $P \to \infty$.
\end{proof}

We also need to control the integrals of all the powers of $\cE$ in the following sense,
for which the dependence on $n$ below is crucial for what comes next.

\begin{lemma} \label{lemma:est_invar_meas_enk}
Let $R_\eps= \frac{\eps^{2\beta}}{2}\|\sigma\|_{L^\infty}^2$ and $B_\eps= 4 \eps^{2 \beta}  \|\sigma\|_{L^\infty}^2$. Then, for all integers $n\geq 0$,
\begin{align*}
\int_{\mathcal{X}}\big(\cE(\p, u)\big)^n \mathcal{K}(\p, u) \d \mu_\eps
\leq 2R_\eps B_\eps^n n! \exp (\frac{h(2R_\eps)}{B_\eps }),
\end{align*}
where $h(x)$ is defined in \eqref{eq:def_h}.
\end{lemma}

\begin{proof}
We first note that, for all $(\p, u) \in \X$,
\begin{align*}
    \mathcal{R}(\p, u)
    &= \sum_{k=1}^\infty \frac{\eps^{2\beta}}{2} \intx \p \sigma_k^2 \d x
    \leq  \sum_{k=1}^\infty \frac{\eps^{2\beta}}{2} \|\sigma_k\|_{L^\infty}^2  \intx \p \d x \leq  \frac{\eps^{2\beta}}{2} \|\sigma\|_{L^\infty}^2 = R_\eps,
\end{align*}
and
\begin{align*}
\mathcal{M}(\p, u)^2
&\leq  \eps^{2 \beta} \sum_{k=1}^\infty \|\sigma_k\|_{L^\infty}^2
   \Big(2 \intx \p \d x \intx \p u^2 \d x + \frac{1}{2}\intx (\log \p)_x^2\Big) \\
     &\leq 4 \eps^{2 \beta}  \|\sigma\|_{L^\infty}^2
       \Big(\frac{1}{2}\|u_x\|^2 + \frac{1}{2 \eps^2}\|(\log \p)_x\|_{L^2}^2\Big)
       = B_\eps \mathcal{K}(\p, u).
\end{align*}
Now, taking $F(z)\equiv 1$ in \eqref{eq:invar_meas_est_F} yields
\begin{equation}\label{eq:invar_meas_est_K}
\int_{\X} \mathcal{K}(\p, u) \d \mu_\eps \leq \int_{\X} \mathcal{R}(\p, u) \d \mu_\eps \leq R_\eps.
\end{equation}
Next, for $n \in  \mathbb{N}$, taking $F(z)=z^n$ in \eqref{eq:invar_meas_est_F}, we have
\begin{align*}
& \int_{\X} \big(\cE(\p, u)\big)^n \mathcal{K}(\p, u) \d \mu_\eps \\
&\leq \int_{\X} \big(\cE(\p, u)\big)^n \mathcal{R}(\p, u) \d \mu_\eps
 + \frac{n}{2} \int_{\X} \big(\cE(\p, u)\big)^{n-1} \big(\mathcal{M}(\p, u)\big)^2 \d \mu_\eps\\
&\leq R_\eps \int_{\X} \big(\cE(\p, u)\big)^n  \d \mu_\eps
  + \frac{B_\eps n}{2} \int_{\X} \big(\cE(\p, u)\big)^{n-1} \mathcal{K}(\p, u) \d \mu_\eps \\
&\leq R_\eps \int_{\X \cap \{\mathcal{K}(\p, u) \leq 2R_\eps\}}
  \big(\cE(\p, u)\big)^n  \d \mu_\eps
  + \int_{\X \cap \{\mathcal{K}(\p, u) > 2R_\eps\}} \big(\cE(\p, u)\big)^n  \d \mu_\eps \\
&\quad + \frac{B_\eps n}{2} \int_{\X} \big(\cE(\p, u)\big)^{n-1} \mathcal{K}(\p, u) \d \mu_\eps \\
&\leq R_\eps \big(h(2R_\eps)\big)^n + \frac{1}{2}\int_{\X \cap \{\mathcal{K}(\p, u) > 2R_\eps\}} \big(\cE(\p, u)\big)^n \mathcal{K}(\p, u)  \d \mu_\eps
+ \frac{B_\eps n}{2} \int_{\X} \big(\cE(\p, u)\big)^{n-1} \mathcal{K}(\p, u) \d \mu_\eps.
\end{align*}
The last inequality is due to the control: $\cE(\p, u) \leq h(\mathcal{K}(\p, u))$.
This implies
\begin{align*}
\int_{\bar{\mathcal{X}}} \big(\cE(\p, u)\big)^n \mathcal{K}(\p, u) \d \mu_\eps
\leq 2R_\eps \big(h(2R_\eps)\big)^n
+ B_\eps n \int_{\bar{\mathcal{X}}} \big(\cE(\p, u)\big)^{n-1} \mathcal{K}(\p, u) \d \mu_\eps.
\end{align*}
Iterating the last inequality and using \eqref{eq:invar_meas_est_K}, we obtain
\begin{align*}
\int_{\bar{\mathcal{X}}} \big(\cE(\p, u)\big)^n \mathcal{K}(\p, u) \d \mu_\eps
\leq 2R_\eps B_\eps^n n! \sum_{k=1}^n \frac{1}{k!} \big(\frac{h(2R_\eps)}{B_\eps }\big)^k
+ R_\eps B_\eps^n n!
\leq 2R_\eps B_\eps^n n! \exp(\frac{h(2R_\eps)}{B_\eps }).
\end{align*}
\end{proof}

The previous lemma allows us to prove new moment bounds of the invariant measures of system \eqref{eq:NS}.
These are fundamental to prove the subsequent convergence results for the invariant measures.

\begin{proposition} \label{prop_invariant_measures_moment}
There exist $C, \lambda>0$, depending only on $\|\sigma\|_{L^\infty}$, such that
\begin{enumerate}
\item[\rm (i)]  For all $p \in [1, \infty)$ and $\eps \in (0, 1]$,  the following polynomial moment bound holds{\rm :}
\begin{equation}\label{eq:NSinvmsrpoly}
\int_{\X} \big(\cE(\p, u)\big)^p \d \mu_\eps \leq C\, p!\, \lambda^p \eps^{2\beta p};
\end{equation}
\item[\rm (ii)] For all $\eps \in (0, 1]$ and $0 < \eta <(2\lambda \eps^{2\beta})^{-1}$, the
following exponential moment
bound holds{\rm :}
\begin{equation}\label{eq:NSinvmsrexp}
\int_{\X} \exp\big(\eta (\cE(\p, u))\big) \d \mu_\eps \leq C.
\end{equation}
\end{enumerate}
\end{proposition}

\begin{proof}
For the polynomial bound \eqref{eq:NSinvmsrpoly}, we preliminarily note that it suffices to show the case when $p$ is a positive integer.
In light of Lemma \ref{lemma:est_invar_meas_enk}, we have
\begin{align*}
\int_{\X} \big(\cE(\p, u)\big)^p \d \mu_\eps
&\leq \int_{\X\cap\{\mathcal{K}(\p, u) \leq B_\eps\}} \big(\cE(\p, u)\big)^p \d \mu_\eps
+\int_{\X \cap\{\mathcal{K}(\p, u) > B_\eps\}} \big(\cE(\p, u)\big)^p \d \mu_\eps\\
&\leq \big(h(B_\eps)\big)^p
+\frac{1}{B_\eps}\int_{\X \cap\{\mathcal{K}(\p, u) > B_\eps\}}\big(\cE(\p, u)\big)^p \mathcal{K}(\p, u) \d\mu_\eps\\
&\leq \big(h(B_\eps)\big)^p + 2 R_\eps B_\eps^{p-1} p! \exp (\frac{h(2R_\eps)}{B_\eps }).
\end{align*}
Since $R_\eps= \frac{\eps^{2\beta}}{2} \|\sigma\|_{L^\infty}^2$
and $B_\eps= 4 \eps^{2 \beta}  \|\sigma\|_{L^\infty}^2$,
we see that $0< B_\eps \leq 4  \|\sigma\|_{L^\infty}^2$ and $0 < R_\eps \leq \frac{1}{2}\|\sigma\|_{L^\infty}^2$,
so that
$0 \leq h(B_\eps)\leq C B_\eps$ and $0 \leq h(2R_\eps)\leq C R_\eps$ for some $C >0$ independent of $\eps \in (0, 1]$.
In addition,  $\frac{h(2R_\eps)}{B_\eps}$ is bounded uniformly in $\eps \in (0, 1]$
so that \eqref{eq:NSinvmsrpoly} follows.

For the exponential moment bound  \eqref{eq:NSinvmsrexp}, we use the polynomial moment bounds and obtain
\begin{align*}
\int_{\bar{\mathcal{X}}} \exp\big(\eta \cE(\p, u)\big) \d \mu_\eps
\leq \sum_{n=0}^\infty  \frac{\eta^n }{n!} \int_\mathcal{X}\big(\cE(\p, u)\big)^n \d \mu_\eps
\leq \sum_{n=0}^\infty  C\eta^n \lambda^n \eps^{2 \beta n}= \frac{C}{1-\eta \lambda \eps^{2\beta}} \leq 2C,
\end{align*}
provided that $0< \eta < (2\lambda \eps^{2\beta})^{-1}$. This completes the proof.
\end{proof}

\subsection{Proof of  Theorem \ref{thm:convergence_invariant_measures}}
To prove Theorem \ref{thm:convergence_invariant_measures},
we follow the strategy of \cite{convergenceMHD}.
Let $T > 0$ to be determined later.
For each $\eps \in (0, 1]$, let $(\bpv, \buv)$ be a statistically stationary solution
of system \eqref{eq:nonlinear} such that $(\bpv(t, \cdot) , \buv(t, \cdot))$ is distributed
as $\bmv$ for $t \geq 0$,
and let $(\rv, \vv)$ be the solution of system \eqref{linear} with initial conditions:
\begin{equation*}
\rv(0, x)= \bpv(0, x), \quad \vv(0, x) = \buv(0, x)  \qquad \text{ for all } x\in (0,1).
\end{equation*}
By Proposition \ref{prop_invariant_measures_moment}, $(\bpv(0, \cdot), \buv(0, \cdot))$ satisfies \eqref{eq:assumption_moment_initial_poly}
with $\alpha = 2 \beta >1$ so that
Theorem \ref{thm:finite_time_convergence} applies.
As a result, we have the stability estimate:
\begin{align*}
 &\mathcal{W}_\eps (\bar{\mu}_\eps, Q_{T}^*\bar{\mu}_\eps) \\
 &\leq \ex \Big[ \big( \|(\rv(T)-\bpv(T),\,\vv(T)-\buv(T))\|_{L^2}^2
 + \eps^2 \|((\rv)_x(T)-(\bpv)_x(T),\,(\vv)_x(T)-(\buv)_x(T))\|_{L^2}^2 \big)^{\frac{1}{2}}\Big]\\
&\leq  C \eps^{2\beta-1} + C \eps \ex \big[ \|((\rv)_x(T, \cdot),\,(\vv)_x(T, \cdot))\|_{L^2}\big] +
C \eps \ex \big[ \|((\bpv)_x(T, \cdot),\,(\buv)_x(T, \cdot))\|_{L^2}\big].
\end{align*}
Now, since the function: $(r,v) \mapsto \eps\|(r_x,\,v_x)\|_{L^2})$ is Lipschitz with
unit Lipschitz constant with respect to norm $\|\cdot\|_\eps$,
the Monge–Kantorovich duality \eqref{def_monge} implies
\begin{align*}
 &\eps \ex \big[ \|((\rv)_x(T),(\vv)_x(T))\|_{L^2} \Big] \\
&= \Big(\int_\mathcal{X} \eps\|(r_x, v_x)\|_{L^2}
\d (\mathcal{Q}_t^{\eps})^*\bar{\mu}_\eps - \int_\mathcal{X} \eps\|(r_x, \,v_x)\|_{L^2} \d \omega_\eps\Big)
   +  \int_\mathcal{X} \eps\|(r_x,\,v_x)\|_{L^2} \d \omega_\eps \\
&\leq \mathcal{W}_\eps ((\mathcal{Q}_t^{\eps})^*\bar{\mu}_\eps, \omega_\eps)
 + C \eps \Big(\int_\mathcal{X} \|(r_x,\,v_x)\|_{L^2}^2\d \omega_\eps \Big)^{\frac{1}{2}}\\
&\leq C \e^{-\frac{t}{4}} \mathcal{W}_\eps (\bar{\mu}_\eps, \omega_\eps) + C \eps^\beta.
\end{align*}

The last inequality is due to the invariance of $\omega_\eps$
and Propositions \ref{prop:exp_contraction} and  \ref{prop:est_invar_measure_linear}.
Moreover, we have
\begin{align*}
 \eps \ex \big[ \|(\bpv)_x(T, \cdot)\|_{L^2}\big]
&=\eps \int_{\bar{\mathcal{X}}_\eps} \|\bp_x\|_{L^2} \d \bmv \\
&\leq \frac{1}{\eps} \int_{\bar{\mathcal{X}}_\eps}
 \big\|1+\eps^2\bp\big\|_{L^\infty} \big\|(\log (1+\eps^2\bp))_x\big\|_{L^2} \d \bmv\\
&\leq \frac{1}{\eps} \Big(\int_{\bar{\mathcal{X}}_\eps}
   \big\|1+\eps^2\bp\big\|_{L^\infty}^2 \d \bmv \Big)^{\frac{1}{2}}
   \Big( \int_{\bar{\mathcal{X}}_\eps}\big\|(\log (1+\eps^2\bp))_x\big\|_{L^2}^2 \d \bmv\Big)^{\frac{1}{2}}\\
&\leq \frac{1}{\eps} \Big(\int_{\bar{\mathcal{X}}_\eps}\exp\big((4\sqrt{2\cE(1+\eps^2 \bp,\eps \bu)}\big)
  \d \bmv \Big)^{\frac{1}{2}}
  \Big( \int_{\bar{\mathcal{X}}_\eps}\big\|(\log (1+\eps^2\bp))_x\big\|_{L^2}^2 \d \bmv \Big)^{\frac{1}{2}}\\
&\leq \frac{C}{\eps}
\Big(\int_{\bar{\mathcal{X}}_\eps}\exp\big(16\cE(1+\eps^2 \bp, \eps\bu)\big)\d \bmv \Big)^{\frac{1}{2}}
\Big(\int_{\bar{\mathcal{X}}_\eps}\big\|(\log (1+\eps^2\bp))_x\big\|_{L^2}^2 \d \bmv \Big)^{\frac{1}{2}}\\
&= \frac{C}{\eps}
\Big(\int_{\X}\exp\big(16\cE(\p, u)\big)\d \mu_\eps  \Big)^{\frac{1}{2}}
\Big( \int_{\X}\big\|(\log (\p))_x\big\|_{L^2}^2 \d \mu_\eps  \Big)^{\frac{1}{2}}\\
&\leq C \eps^{\beta}
\end{align*}
for all sufficiently small $\eps >0$, by virtue of Lemma \ref{lemma:est_invar_meas_enk} and Proposition \ref{prop_invariant_measures_moment}.
Likewise, Proposition \ref{prop_invariant_measures_moment} also gives
\begin{align*}
 \eps \ex \big[ \big\|(\buv)_x(T, \cdot)\big\|_{L^2}\big]
 =\eps \int_{\bar{\mathcal{X}}_\eps} \big\|(\buv)_x\big\|_{L^2} \d \bmv
 =\frac{1}{\eps}\int_{\mathcal{X}} \big\|(\uv)_x\big\|_{L^2} \d \mu_\eps \leq C \eps^{\beta}.
\end{align*}
Finally, combining all these and making use of the triangle inequality and the invariance of $\omega_\eps$,
we have
\begin{align*}
\mathcal{W}_\eps (\bar{\mu}_\eps, \omega_\eps)
&= \mathcal{W}_\eps (\bar{\mu}_\eps, Q_{T}^*\omega_\eps)\\
&\leq \mathcal{W}_\eps (\bar{\mu}_\eps, Q_{T}^*\bar{\mu}_\eps)
+ \mathcal{W}_\eps (Q_{T}^*\bar{\mu}_\eps, Q_{T}^*\omega_\eps)\\
&\leq C \eps^{2\beta-1}+C \e^{-\frac{t}{4}} \mathcal{W}_\eps (\bar{\mu}_\eps, \omega_\eps)
 + C \eps^\beta + C\e^{-\frac{t}{4}}\mathcal{W}_\eps (\bar{\mu}_\eps, \omega_\eps),
\end{align*}
where we have used again Proposition \ref{prop:exp_contraction} in the last inequality.
By choosing $T$ large enough such that $C\e^{-\frac{t}{4}} < \frac{1}{4}$ gives
\begin{align*}
\mathcal{W}_{L^2} (\bar{\mu}_\eps, \omega_\eps) \leq \mathcal{W}_\eps (\bar{\mu}_\eps, \omega_\eps) \leq  C \eps^{2\beta-1}.
\end{align*}
This is the statement of Theorem \ref{thm:convergence_invariant_measures}, and hence the proof is completed.

\bigskip
\medskip
\noindent
{\bf Acknowledgments.}
The research of Gui-Qiang G. Chen was supported in part by the UK Engineering and Physical
Sciences Research Council Awards EP/L015811/1, EP/V008854, and EP/V051121/1.
The research of Michele Coti Zelati was supported in part by the Royal Society URF\textbackslash R1\textbackslash 191492
and EPSRC Horizon Europe Guarantee EP/X020886/1.
The doctoral research of Chin Ching Yeung was supported in part
by scholarships awarded by the Croucher Foundation Limited,
the Centaline Charity Fund, and the China Oxford Scholarship Fund.

\bigskip
\bibliographystyle{abbrv}
\bibliography{zero-mach_arXiv}
\end{document}